\documentclass[reqno]{amsart}
\usepackage[left=2.6cm,right=2.6cm,top=3cm,bottom=3cm]{geometry}
\usepackage[backref=page]{hyperref}
\usepackage{cite}
\usepackage{enumitem}
\usepackage{xcolor} 
\long\def\new#1{#1} 
\long\def\rev#1{#1} 

\usepackage{tikz,tkz-euclide,tkz-tab}

\theoremstyle{plain}
\newtheorem{theorem}{Theorem}
\newtheorem{proposition}[theorem]{Proposition}

\newtheorem{lemma}[theorem]{Lemma}

\theoremstyle{definition}

\theoremstyle{remark}
\newtheorem{remark}{Remark}

\DeclareMathOperator*{\essinf}{ess\,inf}
\DeclareMathOperator*{\esssup}{ess\,sup}

\allowdisplaybreaks

\begin{document}
	
	\author{Phuong Le}
	\address{Phuong Le$^{1,2}$ (ORCID: 0000-0003-4724-7118)\newline
		$^1$Faculty of Economic Mathematics, University of Economics and Law, Ho Chi Minh City, Vietnam; \newline
		$^2$Vietnam National University, Ho Chi Minh City, Vietnam}
	\email{phuongl@uel.edu.vn}
	
	
	
	
	\title[Sharp boundary asymptotics and monotonicity]{Sharp boundary asymptotics and monotonicity for singular quasilinear elliptic equations with gradient terms in half-spaces}
	\begin{abstract}
We establish sharp boundary asymptotics for directional derivatives and
monotonicity of positive weak solutions to
$-\Delta_p u+\vartheta|\nabla u|^q=u^{-\gamma}+f(u)$ in the half-space
$\mathbb R^N_+$ with zero Dirichlet data, where $p>1$, $\gamma>0$,
$0<q\le p$, $\vartheta\ge0$, and $f$ is locally Lipschitz continuous.
For solutions bounded in a boundary strip, writing $t=x_N\downarrow0$,
we prove the exact limits
\[
\begin{aligned}
t^{\frac{\gamma-1}{\gamma+p-1}}\,\partial_\eta u(x',t)
&\longrightarrow \frac{p C_{\gamma,p}}{\gamma+p-1}\,\langle\eta,e_N\rangle
&& (\gamma>1),\\
(1-\log t)^{-\frac1p}\,\partial_\eta u(x',t)
&\longrightarrow \left(\frac{p}{p-1}\right)^{\frac1p}
\langle\eta,e_N\rangle && (\gamma=1),
\end{aligned}
\]
where
$C_{\gamma,p}=\bigl((\gamma+p-1)^p/[p^{p-1}(p-1)(\gamma-1)]\bigr)^{1/(\gamma+p-1)}$.
The convergence is uniform in the tangential variable; in particular,
the appropriately rescaled tangential derivatives vanish and the leading
constants depend only on $p$ and $\gamma$, not on the gradient term or $f$.
The normal-derivative limits imply the corresponding exact leading-order
boundary profiles of $u$ by integration. In the weakly singular regime
$0<\gamma<1$, we obtain linear boundary bounds and positive inward
directional-derivative estimates. We also establish local counterparts of
the boundary results for solutions merely locally bounded up to the
boundary. Finally, under additional structural assumptions, the moving-plane
method yields normal monotonicity for strip-bounded solutions in arbitrary
dimension and for locally bounded solutions in dimension two. The boundary
asymptotics hold throughout the range $0<q\le p$.
\end{abstract}
	
	\maketitle
	
	\section{Introduction}
	We are interested in the qualitative properties of solutions to the quasilinear elliptic problem
	\begin{equation}\label{main}
		\begin{cases}
			-\Delta_p u + \vartheta |\nabla u|^q = \frac{1}{u^\gamma} + f(u) &\text{ in } \mathbb{R}^N_+,\\
			u>0 &\text{ in } \mathbb{R}^N_+,\\
			u=0 &\text{ on } \partial\mathbb{R}^N_+,
		\end{cases}
	\end{equation}
	where $N\ge2$, $p>1$, $\gamma>0$, $0<q\le p$, $\vartheta\ge0$, $f:[0,+\infty)\to\mathbb{R}$ is a locally Lipschitz continuous function and $\mathbb{R}^{N}_{+}:=\{x=(x',x_N)\in\mathbb{R}^{N} \mid x'\in\mathbb{R}^{N-1},x_{N}>0\}$ is the upper half-space. We call $u \in W^{1,p}_{\rm loc}(\mathbb{R}^N_+)$ a \textit{weak solution} to problem \eqref{main} if
	\[
	\essinf_{x\in K} u(x) > 0 \quad \text{ for every } K \subset\subset \mathbb{R}^N_+,
	\]
	and $f(u)\in L^1_{\rm loc}(\mathbb R^N_+)$, with all terms in the following identity finite for every test function, and
		\[
		\int_{\mathbb{R}^N_+} |\nabla u|^{p-2} \langle \nabla u, \nabla \varphi \rangle + \vartheta \int_{\mathbb{R}^N_+} |\nabla u|^q \varphi = \int_{\mathbb{R}^N_+} \left(\frac{1}{u^\gamma} + f(u)\right) \varphi \quad\text{ for } \varphi \in C^1_c(\mathbb{R}^N_+).
	\]
	Moreover, the zero Dirichlet boundary condition also has to be understood in the weak meaning
	\[
	(u - \varepsilon)^+ \varphi\big|_{\mathbb{R}^N_+} \in W^{1,p}_0(\mathbb{R}^N_+) \quad \text{ for every nonnegative } \varphi \in C^1_c(\mathbb{R}^N) \text{ and } \varepsilon>0.
	\]
	
	Our first aim is to identify sharp boundary scales and directional-derivative asymptotics for solutions to problem \eqref{main}. Our second aim is to establish monotonicity in the $x_N$-direction. For the latter, we use the moving plane method, which was introduced by Alexandrov \cite{MR143162} and Serrin \cite{MR333220}. This method relies on several types of maximum principles and was used in the celebrated papers \cite{MR1470317,MR544879}, where the symmetry of solutions to semilinear elliptic equations in bounded domains was established. For semilinear elliptic equations $-\Delta u=f(u)$ in $\mathbb{R}^{N}_{+}$, where the monotonicity of the solutions is expected, we refer to \cite{MR1655510,MR1470317,MR2076772,MR1190345,MR2525168,2024arXiv240900365L}.
	
	The quasilinear elliptic equations involving the $p$-Laplacian are harder to deal with because of the nonlinear nature of the operator. Hence, comparison principles are not equivalent to maximum principles for the $p$-Laplacian. Furthermore, the singularity or degeneracy of the operator (corresponding to $1<p<2$ and $p>2$, respectively) also causes the lack of $C^{2}$ regularity of the weak solutions and other difficulties. First symmetry results for the regular $p$-Laplace problem
	\begin{equation}\label{regular}
		\begin{cases}
			-\Delta_{p}u=f(u) & \text{ in }\Omega, \\
			u>0 & \text{ in }\Omega, \\
			u=0 & \text{ on }\partial\Omega,
		\end{cases}
	\end{equation}
	where $\Omega$ is a bounded domain of $\mathbb{R}^{N}$, were obtained in \cite{MR1648566} for the case $1<p<2$. The case $p>2$ requires the use of weighted Sobolev spaces and was solved in \cite{MR2096703} for positive nonlinearities.
	The monotonicity of (possibly unbounded) solutions to \eqref{regular} with $\Omega=\mathbb{R}_{+}^{N}$ was established for $1<p<2$ in \cite{MR3303939} and for $p>2$ in \cite{MR3752525} via the moving plane method. All these results assume that $f$ is a positive locally Lipschitz continuous function. More recently, \cite{MR3303939} was partially improved by the work \cite{MR4439897}, where the restriction on positivity of $f$ was relaxed.
	
	Due to the presence of the term $\frac{1}{u^\gamma}$ with $\gamma>0$, problem \eqref{main} is singular. For $\gamma\ge1$ the gradient has a singular boundary scale and generally blows up, whereas the regime $0<\gamma<1$ admits boundary $C^1$ regularity under the assumptions below (see \cite{MR1037213,2025arXiv250807360L}). Furthermore, the difficulty in understanding problem \eqref{main} lies in the fact that both sides of the problem may exhibit singularity near $\partial\Omega$.
	Singular elliptic problems without a gradient term have been studied by various authors in the past. We refer in particular to the pioneering papers by Crandall, Rabinowitz and Tartar \cite{MR427826} and by Lazer and McKenna \cite{MR1037213}, where they combine the method of subsolutions
	and supersolutions with approximation technique to prove the existence, uniqueness and qualitative properties of solutions to
	\begin{equation}\label{pure_singular}		
		\begin{cases}
			-\Delta_p u = \frac{a(x)}{u^\gamma} &\text{ in } \Omega,\\
			u>0 &\text{ in } \Omega,\\
			u=0 &\text{ on } \partial\Omega
		\end{cases}
	\end{equation}
	with $p=2$ (see \cite{MR1263903} for further insights). We also refer to the work by Boccardo and Orsina \cite{MR2592976} on the same problem in a more general setting, exploiting truncation arguments and Schauder's fixed point theorem. Later, the existence, uniqueness and regularity of solutions to problem \eqref{pure_singular} with $p>1$ were derived in \cite{MR3136107,MR3478284}. The Hopf type lemma for singular quasilinear equations was derived in \cite{MR4044739}. The boundary behavior of such problems with gradient terms in bounded domains was recently studied by the author in \cite{2025arXiv250807360L}.
	
	For the gradient-free problem ($\vartheta=0$), earlier monotonicity results in the strongly singular regime $\gamma>1$ include \cite[Theorem 3]{2024arXiv240919557L} for $p>1$ and \cite{2024arXiv240900365L} for $p=2$. More precisely, among other things, the following result was obtained in \cite[Theorem 3]{2024arXiv240919557L}: \textit{Let $u$ be a solution to	
	\begin{equation}\label{no-grad}
		\begin{cases}
			-\Delta_p u = \frac{1}{u^\gamma} + f(u) &\text{ in } \mathbb{R}^N_+,\\
			u>0 &\text{ in } \mathbb{R}^N_+,\\
			u=0 &\text{ on } \partial\mathbb{R}^N_+,
		\end{cases}
	\end{equation}
	where $\gamma>1$, such that $u$ is bounded in strips
	\[
	\Sigma_\lambda := \{x\in\mathbb{R}^N \mid 0<x_N<\lambda\}
	\]
	for any $\lambda>0$. If we further assume that either $\frac{2N+2}{N+2} < p < 2$ and $\{t\in(0,+\infty) \mid \frac{1}{t^\gamma} + f(t)=0\}$ is a discrete set, or $\frac{1}{t^\gamma} + f(t)>0$ for $t>0$, then $u$ is monotone increasing in $x_N$. Moreover, we have $\frac{\partial u}{\partial x_N} > 0$ in $\mathbb{R}^N_+$ provided that $p>\frac{2N+2}{N+2}$ and $\frac{1}{t^\gamma} + f(t)>0$ for $t>0$.} Montoro, Muglia, and Sciunzi \cite{MontoroMugliaSciunzi2025NoDEA} also obtained monotonicity for the gradient-free equation in the regime $\gamma>1$, assuming positivity of $t^{-\gamma}+f(t)$ and boundedness of the solution in strips.
		
	The purpose of the present work is twofold. First, independently of any global monotonicity conclusion, we determine the exact leading-order limits of all rescaled directional derivatives in both regimes $\gamma>1$ and $\gamma=1$, together with sharp boundary information in the weakly singular range $0<\gamma<1$. These derivative limits capture not only the normal blow-up rate and its explicit universal coefficient but also the disappearance of the rescaled tangential components. The corresponding sharp asymptotics for the solution itself follow by integration; in the borderline regime, the logarithmic barrier result of Theorem~\ref{th:sharp=1} is established independently and used to identify the derivative limit. Second, we apply the boundary information to obtain monotonicity results allowing the gradient term $\vartheta|\nabla u|^q$ and, in dimension two, solutions that need not be bounded in strips. Thus the contribution is not restricted to the gradient-free, strongly singular monotonicity setting of \cite{2024arXiv240919557L,MontoroMugliaSciunzi2025NoDEA}. The main advances and technical issues are as follows:
	\begin{itemize}
		\item When $\gamma > 1$ and $\vartheta=0$, the following sharp bounds can be established for solutions that are bounded in some strip $\Sigma_{\lambda_0}$:
		\begin{equation}\label{bounds1}
			c_1 x_{N}^{\frac{p}{\gamma + p - 1}} \leq u(x) \leq c_2 x_{N}^{\frac{p}{\gamma + p - 1}} \quad \text{in } \Sigma_{\lambda_0},
		\end{equation}
		where $c_2 > c_1 > 0$ (see \cite[Lemmas 16 and 17]{2024arXiv240919557L} for $p > 1$ and \cite{MR4753083} for $p = 2$). 
		We emphasize that such sharp estimates are crucial for determining the behavior of the directional derivatives of solutions near the boundary (see Theorem~\ref{th:boundary_behavior}). 
		However, the estimates in \eqref{bounds1} do not hold in the regime $0 < \gamma \leq 1$. 
		More precisely, while the lower bound in \eqref{bounds1} remains valid for $0 < \gamma \leq 1$, the upper bound cannot be obtained using the techniques of \cite{2024arXiv240919557L,MR4753083}. 
		In the present work, by introducing new barrier functions, we derive the following sharp bounds:
		\begin{equation}\label{bounds2}
			c_1 x_N \left(1 - \ln x_N \right)^{\frac{1}{p}} \leq u(x) \leq c_2 x_N \left(1 - \ln x_N \right)^{\frac{1}{p}} \quad \text{in } \Sigma_{\lambda_0}
		\end{equation}
		for the case $\gamma = 1$ (assuming $\lambda_0 < 1$ in this case), and 
		\begin{equation}\label{bounds3}
			c_1 x_N \leq u(x) \leq c_2 x_N \quad \text{in } \Sigma_{\lambda_0}
		\end{equation}
		for the case $0 < \gamma < 1$. 
		These results reveal a completely different behavior of solutions compared to the case $\gamma > 1$.
		\new{Furthermore, in the borderline case $\gamma=1$ we are able to compute the sharp constant: we show in Theorem~\ref{th:sharp=1} that
		\[
		\lim_{x_N\to0^+} \frac{u(x)}{x_N\left(1-\ln x_N\right)^{\frac{1}{p}}} = \left(\frac{p}{p-1}\right)^{\frac{1}{p}}
		\]
		uniformly with respect to $x'$, so that $c_1$ and $c_2$ in \eqref{bounds2} can be taken arbitrarily close to $\left(\frac{p}{p-1}\right)^{1/p}$. This is the exact analogue of the classification result available for $\gamma>1$ and it is what allows us to identify the limit of the rescaled gradients in \eqref{limit=1} below.}
		
		We stress that our proof does not assume $u \in C^1$, but only $u \in W^{1,p}_{\mathrm{loc}}(\mathbb{R}^N_+)$. 
		Nevertheless, combining these estimates with standard regularity results yields that $u$ is actually $C^1$ in the interior of \rev{the strip} and continuous up to the boundary. 
		We believe that our new barrier functions may be applied in other contexts as well; for instance, they could be used to refine the estimates in \cite{MR4833380} for $0 < \gamma \leq 1$.
		
		\item Another delicate issue arises from the presence of the possibly unbounded gradient term $|\nabla u|^q$ with $0 < q \leq p$ when $\vartheta>0$. 
		For regular problems with gradient terms, such as those considered in \cite{MR3303939,MR3752525,MR4439897,MR3118616}, these terms are bounded by standard regularity and can be handled using appropriate estimation techniques. 
		However, when such terms become unbounded, most results in the literature restrict to the case $q = p$, allowing for the use of a suitable test function to eliminate these terms (see \cite{MR3165277,MR4333974} and related works). 
		In contrast, by exploiting a recently established comparison principle, we are able to treat the full range $0 < q \leq p$ in our \rev{boundary estimates (Theorems \ref{th:boundary_behavior} and \ref{th:boundary_behavior'} below, together with Theorems \ref{th:boundary_bound} and \ref{th:sharp=1} in Section 2)}, regardless of the potential unboundedness of $|\nabla u|$ near the boundary of $\mathbb{R}^N_+$. \rev{In the monotonicity results we need, as in the regular case, $q>1$ when $p\le2$ and $q\ge p-1$ when $p>2$.}
		
		\item Finally, we extend our analysis to solutions that may be unbounded in strips. 
		We obtain local versions of estimates \eqref{bounds1}, \eqref{bounds2}, and \eqref{bounds3}, which characterize the local behavior of solutions near the boundary. 
		Using these local estimates in conjunction with the sliding and rotating lines method, we deduce the monotonicity of solutions in dimension two without requiring boundedness in strips. 
		\rev{The sharp constant in the borderline case $\gamma=1$ survives in this local setting.}
\end{itemize}
The boundary-scale and derivative estimates below, including the sharp
borderline constant, are established without the additional restrictions on
$p$, $q$, and $t^{-\gamma}+f(t)$ used in the monotonicity theorems. They are
therefore separate asymptotic results, not merely auxiliary statements in
the moving-plane argument.
To state our first result, we denote by $B_r(x_0)$ the open ball in $\mathbb{R}^N$ of center $x_0$ and radius $r$. We also write $B_r=B_r(0)$.
	Then we denote $\mathbb{S}^{N-1}:=\partial B_1(0)$ and $e_N:=(0,\dots,0,1)$.
	
	We start with a result on the behavior of solutions near the boundary.
	\begin{theorem}\label{th:boundary_behavior} Let $u \in W^{1,p}_{\rm loc}(\mathbb{R}^N_+)$ be a solution to problem \eqref{main} with $u\in L^\infty(\Sigma_{\overline\lambda})$ for some $\overline\lambda>0$. Then the following assertions hold:
		\begin{enumerate}
			\item[(i)] When $\gamma>1$, we have $u \in C^{1,\alpha}_{\rm loc}(\rev{\Sigma_{\overline\lambda}})$ for some $\alpha\in(0,1)$ \rev{and $u\in C^{0,\frac{p}{\gamma+p-1}}(\overline{\Sigma_{\lambda_0}})$ for some $\lambda_0>0$}. Moreover, for any $\eta_0 \in \mathbb{S}^{N-1}$,
			\begin{equation}\label{limit>1}
				\lim_{\substack{x_N\to0^+\\\eta\to\eta_0}} x_N^\frac{\gamma-1}{\gamma+p-1} \frac{\partial u(x)}{\partial \eta} = \left(\frac{(\gamma + p - 1)^p}{p^{p-1}(p-1)(\gamma - 1)}\right)^{\frac{1}{\gamma + p - 1}} \frac{p \langle\eta_0,e_N\rangle}{\gamma + p - 1}.
			\end{equation}
			Consequently, for every $\beta\in(0,1)$, there exist $\lambda_0,c_1,c_2>0$ such that
			\begin{equation}\label{derivative_bound>1}
				c_1 x_N^\frac{1-\gamma}{\gamma+p-1} < \frac{\partial u(x)}{\partial \eta} < c_2 x_N^\frac{1-\gamma}{\gamma+p-1} \quad\text{ in } \Sigma_{\lambda_0}
			\end{equation}
			for all $\eta\in \mathbb{S}^{N-1}$ with $\langle\eta,e_N\rangle \ge \beta$.
			\item[(ii)] When $\gamma=1$, we have $u \in C^{1,\alpha}_{\rm loc}(\rev{\Sigma_{\overline\lambda}})$ for some $\alpha\in(0,1)$ \rev{and $u\in C^{0,s}(\overline{\Sigma_{\lambda_0}})$ for some $\lambda_0>0$ and all $s\in(0,1)$}. \new{Moreover, for any $\eta_0\in\mathbb{S}^{N-1}$,
			\begin{equation}\label{limit=1}
				\lim_{\substack{x_N\to0^+\\\eta\to\eta_0}} \left(1-\ln x_N\right)^{-\frac{1}{p}} \frac{\partial u(x)}{\partial \eta} = \left(\frac{p}{p-1}\right)^{\frac{1}{p}} \langle\eta_0,e_N\rangle.
			\end{equation}
			Consequently,} for every $\beta\in(0,1)$, there exist $0<\lambda_0<1$, $c_1,c_2>0$ such that
			\begin{equation}\label{derivative_bound=1}
				c_1 \left(1-\ln x_N\right)^\frac{1}{p} \le \frac{\partial u(x)}{\partial \eta} \le c_2 \left(1-\ln x_N\right)^\frac{1}{p} \quad\text{ in } \Sigma_{\lambda_0}
			\end{equation}
			for all $\eta\in \mathbb{S}^{N-1}$ with $\langle\eta,e_N\rangle \ge \beta$.
			\item[(iii)] When $0<\gamma<1$, we have $u \in C^{1,\alpha}_{\rm loc}(\rev{\Sigma_{\overline\lambda}\cup\partial\mathbb{R}^N_+})$ for some $\alpha\in(0,1)$. Moreover, for every $\beta\in(0,1)$, there exist $\lambda_0,c_1,c_2>0$ such that
			\begin{equation}\label{derivative_bound<1}
				c_1 < \frac{\partial u(x)}{\partial \eta} < c_2 \quad\text{ in } \Sigma_{\lambda_0}\cup\partial\mathbb{R}^N_+
			\end{equation}
			for all $\eta\in \mathbb{S}^{N-1}$ with $\langle\eta,e_N\rangle \ge \beta$.
		\end{enumerate}
	\end{theorem}
	
	\new{Taking $\beta=\frac{1}{3}$ in \eqref{derivative_bound>1}--\eqref{derivative_bound<1} and writing an arbitrary unit vector $\xi$ as
\[
\xi = |e_N+\tfrac{\xi}{2}|\,\eta_+ - |e_N-\tfrac{\xi}{2}|\,\eta_-, \qquad \eta_\pm := \frac{e_N\pm\frac{\xi}{2}}{|e_N\pm\frac{\xi}{2}|} \quad \left(\text{so that } \langle\eta_\pm,e_N\rangle\ge\tfrac13\right),
\]
we obtain the following two-sided control of the \emph{full} gradient: there are $\lambda_0,C>0$ such that
\begin{equation}\label{full_gradient}
	|\nabla u| \le C \times
	\begin{cases}
		x_N^\frac{1-\gamma}{\gamma+p-1} & \text{ if } \gamma>1,\\
		\left(1-\ln x_N\right)^\frac{1}{p} & \text{ if } \gamma=1,\\
		1 & \text{ if } 0<\gamma<1
	\end{cases}
	\qquad \text{ in } \Sigma_{\lambda_0}.
\end{equation}
Estimate \eqref{full_gradient} will be used repeatedly below; in the borderline case it yields the sharp modulus of continuity of $u$, see Remark \ref{rem:loglip}.}

In particular, Theorem \ref{th:boundary_behavior} implies that solutions are increasing in the $x_N$-direction in a narrow strip from the boundary. It provides us with a starting point to carry out the moving plane procedure to derive our next result. Consider the following assumptions:
	\begin{enumerate}[label=($F_\arabic*$)]
		\item\label{f1} $\frac{2N+2}{N+2} < p \le 2$ and $\{t\in(0,+\infty) \mid \frac{1}{t^\gamma} + f(t)=0\}$ is a discrete set,
		\item\label{f2} $p>1$ and $\frac{1}{t^\gamma} + f(t)>0$ for $t>0$.
	\end{enumerate}
	
	We can prove the following monotonicity result.
	\begin{theorem}\label{th:monotonicity} 
		Assume that either \ref{f1} or \ref{f2} holds. We also assume that $1< q\le p$ if $1<p\le2$ and $p-1\le q\le p$ if $p>2$. Let $u \in W^{1,p}_{\rm loc}(\mathbb{R}^N_+)$ be a solution to problem \eqref{main} with $u\in L^\infty(\Sigma_\lambda)$ for all $\lambda>0$. Then $u$ is monotone increasing in $x_N$-direction, that is,
		\[
		\frac{\partial u}{\partial x_N} \ge 0 \quad\text{ in } \mathbb{R}^N_+.
		\]
		Moreover, if $p>\frac{2N+2}{N+2}$ and $\frac{1}{t^\gamma} + f(t)>0$ for $t>0$, then
		\[
		\frac{\partial u}{\partial x_N} > 0 \quad\text{ in } \mathbb{R}^N_+.
		\]
	\end{theorem}
	
	We stress that all above results assume that $u$ is bounded in some strip. Actually, some of our result can be formulated in a weaker form for solutions without this assumption. For $R>0$, we set
	\[
	\mathcal{C}_R = \{x:=(x',x_N)\in\mathbb{R}^N_+ \mid |x'|<R\}.
	\] We have the following local version of Theorem \ref{th:boundary_behavior}, which applies to solutions which may be unbounded in strips.
	\begin{theorem}\label{th:boundary_behavior'} Let $u \in W^{1,p}_{\rm loc}(\mathbb{R}^N_+)\cap L^\infty_{\rm loc}(\overline{\mathbb{R}^N_+})$ be a solution to problem \eqref{main}. Then the following assertions hold:
		\begin{enumerate}
			\item[(i)] When $\gamma>1$, we have $u \in C^{1,\alpha}_{\rm loc}(\mathbb{R}^N_+)\cap C^{\frac{p}{\gamma+p-1}}_{\rm loc}(\overline{\mathbb{R}^N_+})$ for some $\alpha\in(0,1)$. Moreover, for any $x_0\in\partial\mathbb{R}^N_+$ and $\eta_0 \in \mathbb{S}^{N-1}$,
			\[
			\lim_{\substack{\mathbb{R}^N_+\ni x\to x_0\\\eta\to\eta_0}} x_N^\frac{\gamma-1}{\gamma+p-1} \frac{\partial u(x)}{\partial \eta} = \left(\frac{(\gamma + p - 1)^p}{p^{p-1}(p-1)(\gamma - 1)}\right)^{\frac{1}{\gamma + p - 1}} \frac{p \langle\eta_0,e_N\rangle}{\gamma + p - 1}.
			\]
			Consequently, for every $R>0$ and $\beta\in(0,1)$, there exist $\lambda_0,c_1,c_2>0$ such that
			\[
			c_1 x_N^\frac{1-\gamma}{\gamma+p-1} < \frac{\partial u(x)}{\partial \eta} < c_2 x_N^\frac{1-\gamma}{\gamma+p-1} \quad\text{ in } \Sigma_{\lambda_0} \cap \mathcal{C}_R
			\]
			for all $\eta\in \mathbb{S}^{N-1}$ with $\langle\eta,e_N\rangle \ge \beta$.
						\item[(ii)] When $\gamma=1$, we have $u \in C^{1,\alpha}_{\rm loc}(\mathbb{R}^N_+) \cap C^s_{\rm loc}(\overline{\mathbb{R}^N_+})$ for some $\alpha\in(0,1)$ and all $s\in(0,1)$. Moreover, \rev{for any $x_0\in\partial\mathbb{R}^N_+$ and $\eta_0 \in \mathbb{S}^{N-1}$,
				\[
				\lim_{\substack{\mathbb{R}^N_+\ni x\to x_0\\\eta\to\eta_0}} \left(1-\ln x_N\right)^{-\frac{1}{p}} \frac{\partial u(x)}{\partial \eta} = \left(\frac{p}{p-1}\right)^{\frac{1}{p}} \langle\eta_0,e_N\rangle.
				\]
				Consequently,} for every $R>0$ and  $\beta\in(0,1)$, there exist $0<\lambda_0<1$, $c_1,c_2>0$ such that
			\[
			c_1 \left(1-\ln x_N\right)^\frac{1}{p} \le \frac{\partial u(x)}{\partial \eta} \le c_2 \left(1-\ln x_N\right)^\frac{1}{p} \quad\text{ in } \Sigma_{\lambda_0} \cap \mathcal{C}_R
			\]
			for all $\eta\in \mathbb{S}^{N-1}$ with $\langle\eta,e_N\rangle \ge \beta$.
			\item[(iii)] When $0<\gamma<1$, we have $u \in C^{1,\alpha}_{\rm loc}(\overline{\mathbb{R}^N_+})$ for some $\alpha\in(0,1)$. Moreover, for every $R>0$ and  $\beta\in(0,1)$, there exist $\lambda_0,c_1,c_2>0$ such that
			\[
			c_1 < \frac{\partial u(x)}{\partial \eta} < c_2 \quad\text{ in } \left(\Sigma_{\lambda_0}\cup\partial\mathbb{R}^N_+\right) \cap \mathcal{C}_R
			\]
			for all $\eta\in \mathbb{S}^{N-1}$ with $\langle\eta,e_N\rangle \ge \beta$.
		\end{enumerate}
	\end{theorem}
	
	With Theorem \ref{th:boundary_behavior'} in force, we can proceed as in \cite{MR2654242,MR4290571,MR4876247}, exploiting the rotating and moving lines technique to prove the monotonicity of solutions in dimension two without the assumption on their boundedness in strips. The following result improves Theorem \ref{th:monotonicity} in dimension two.
	\begin{theorem}\label{th:monotonicity'}
		Assume that one of the following holds:
		\begin{itemize}
			\item[(i)] $\frac{3}{2} < p \le 2$, $1< q\le p$ and $\{t\in(0,+\infty) \mid \frac{1}{t^\gamma} + f(t)=0\}$ is a discrete set,
			\item[(ii)] $p>2$, $p-1\le q\le p$ and $\frac{1}{t^\gamma} + f(t)>0$ for $t>0$.
		\end{itemize}
		Let $u \in W^{1,p}_{\rm loc}(\mathbb{R}^2_+)\cap L^\infty_{\rm loc}(\overline{\mathbb{R}^2_+})$ be a solution to the problem
		\[
		\begin{cases}
			-\Delta_p u + \vartheta |\nabla u|^q = \frac{1}{u^\gamma} + f(u) &\text{ in } \mathbb{R}^2_+,\\
			u>0 &\text{ in } \mathbb{R}^2_+,\\
			u=0 &\text{ on } \partial\mathbb{R}^2_+.
		\end{cases}
		\]
		Then $u$ is monotone increasing in $x_2$-direction, that is,
		\[
		\frac{\partial u}{\partial x_2} \ge 0 \quad\text{ in } \mathbb{R}^2_+.
		\]
		Moreover,
		\[
		\frac{\partial u}{\partial x_2} > 0 \quad\text{ in } \mathbb{R}^2_+ \setminus \left\{ x \in \mathbb{R}^2_+ \mid \frac{1}{u^\gamma} + f(u)=0\right\}.
		\]
	\end{theorem}
	
	The rest of this paper is devoted to the proofs of the above theorems. In Section 2, we prove sharp bounds for solutions near the boundary via comparison arguments. Then we identify the sharp limits of rescaled directional derivatives for $\gamma\ge1$ in Theorem \ref{th:boundary_behavior} by means of the scaling technique. \rev{The logarithmic barrier construction in Theorem \ref{th:sharp=1} supplies the universal coefficient in the borderline case $\gamma=1$.} In Section 3, we exploit the moving plane method to derive the monotonicity of solutions stated in Theorem \ref{th:monotonicity}. In Section 4, we deal with solutions which may be unbounded in every strip. We demonstrate how to derive Theorem \ref{th:boundary_behavior'} by modifying the proof\rev{s} of Theorem\rev{s} \ref{th:boundary_behavior} \rev{and \ref{th:sharp=1}}. Then we provide a proof for Theorem \ref{th:monotonicity'}.
	
	\section{Boundary behavior of solutions}
	In this section, we study the behavior of solutions to \eqref{main} \new{which} are bounded in some strip.
	\subsection{Sharp bounds for solutions near the boundary}
	The aim of this subsection is the following.
	\begin{theorem}\label{th:boundary_bound}
		Let $u \in W^{1,p}_{\rm loc}(\mathbb{R}^N_+)$ be a solution to problem \eqref{main} with $u\in L^\infty(\Sigma_{\overline\lambda})$ for some $\overline\lambda>0$. Then $u\in C^{1,\alpha}_{\rm loc}(\rev{\Sigma_{\overline\lambda}})$ for some $0<\alpha<1$\rev{, and $u$, extended by zero, is continuous on $\Sigma_{\overline\lambda}\cup\partial\mathbb{R}^N_+$}. Moreover, the following assertions hold:
		\begin{enumerate}
			\item[(i)] When $\gamma>1$, there exist $\lambda_0,c_1,c_2>0$ such that
			\[
			c_1 x_N^\frac{p}{\gamma+p-1} \le u(x) \le c_2 x_N^\frac{p}{\gamma+p-1} \quad\text{ in } \Sigma_{\lambda_0}.
			\]
			\item[(ii)] When $\gamma=1$, there exist $0<\lambda_0<1$, $c_1,c_2>0$ such that
			\[
			c_1 x_N \left(1-\ln x_N\right)^\frac{1}{p} \le u(x) \le c_2 x_N \left(1-\ln x_N\right)^\frac{1}{p} \quad\text{ in } \Sigma_{\lambda_0}.
			\]
			\item[(iii)] When $0<\gamma<1$, there exist $\lambda_0,c_1,c_2>0$ such that
			\[
			c_1 x_N \le u(x) \le c_2 x_N \quad\text{ in } \Sigma_{\lambda_0}.
			\]
		\end{enumerate}
	\end{theorem}
	
	\rev{\begin{remark}\label{rem:interior}
		Theorems \ref{th:boundary_behavior} and \ref{th:boundary_bound} only assume that $u$ is bounded in one strip $\Sigma_{\overline\lambda}$, and their regularity assertions are accordingly stated in $\Sigma_{\overline\lambda}$. Since $f$ may have any growth at infinity, a function in $W^{1,p}_{\rm loc}(\mathbb{R}^N_+)$ solving \eqref{main} need not be locally bounded above $\Sigma_{\overline\lambda}$. If in addition $u\in L^\infty_{\rm loc}(\mathbb{R}^N_+)$, for instance if $u$ is bounded in every strip, then $u\in C^{1,\alpha}_{\rm loc}(\mathbb{R}^N_+)\cap C(\overline{\mathbb{R}^N_+})$ by \cite{MR709038,MR727034}.
	\end{remark}}

	\new{In the borderline case $\gamma=1$ the constants in Theorem \ref{th:boundary_bound}(ii) turn out to be universal. This is the content of our next result, which is new even for $p=2$ and $\vartheta=0$.}

	\begin{theorem}[\new{The sharp constant in the borderline case}]\label{th:sharp=1}
		\new{Let $\gamma=1$ and let $u \in W^{1,p}_{\rm loc}(\mathbb{R}^N_+)$ be a solution to problem \eqref{main} with $u\in L^\infty(\Sigma_{\overline\lambda})$ for some $\overline\lambda>0$. Then
		\[
		\lim_{x_N\to0^+} \frac{u(x)}{x_N\left(1-\ln x_N\right)^\frac{1}{p}} = \left(\frac{p}{p-1}\right)^\frac{1}{p}
		\]
		uniformly with respect to $x'\in\mathbb{R}^{N-1}$.}
	\end{theorem}
	
	\begin{remark}
		\new{The constant $\kappa_*:=\left(\frac{p}{p-1}\right)^{1/p}$ is precisely the one dictated by the one-dimensional profile. Indeed, if $\psi>0$ solves $-\left(|\psi'|^{p-2}\psi'\right)'=\frac{1}{\psi}$ on $(0,\delta)$ with $\psi(0)=0$ and $\psi'>0$, then multiplying by $\psi'$ and integrating gives the first integral
		\[
		\left(\psi'(t)\right)^p = \frac{p}{p-1}\left(C-\ln \psi(t)\right),
		\]
		whence $\psi'(t)\sim\kappa_*\left(-\ln t\right)^{1/p}$ and $\psi(t)\sim\kappa_* t\left(-\ln t\right)^{1/p}$ as $t\to0^+$. Theorem \ref{th:sharp=1} says that every solution of \eqref{main} with $\gamma=1$ is asymptotically flat near $\partial\mathbb{R}^N_+$ and matches this profile, no matter how large $\vartheta$, $q$ and $f$ are. Let us stress that no such universality holds when $0<\gamma<1$: there the singular term is subcritical and the limit profile $a\,x_N$ carries an undetermined constant $a>0$.}
	\end{remark}

We list here two main tools we will exploit to prove Theorem \ref{th:boundary_bound}.
	We recall the following comparison principle from \cite{2025arXiv250807360L} (see also \cite{MR4044739} for the case $\vartheta=0$).
	\begin{lemma}[Lemma 9 in \cite{2025arXiv250807360L}]\label{lem:wcp}
		Let $c>0$ and $\Omega$ be a domain of $\mathbb{R}^N$. Let $u_1,u_2\in C^{1,\alpha}_{\rm loc}(\Omega) \cap L^\infty(\Omega)$ be such that
		\[
		\begin{cases}
			-\Delta_p u_1 + \vartheta |\nabla u_1|^p + \vartheta \le \frac{c}{u_1^\gamma} &\text{ in }  \Omega,\\
			-\Delta_p u_2 + \vartheta |\nabla u_2|^p + \vartheta \ge \frac{c}{u_2^\gamma} &\text{ in }  \Omega,\\
			u_1,u_2 > 0 &\text{ in }  \Omega,\\
			u_1 \le u_2 &\text{ on } \partial \Omega.
		\end{cases}
		\]
		If $\vartheta>0$, we further assume
		\[
		u_1 \le \left( \frac{c}{\vartheta} \right)^{1/\gamma} \quad\text{ in } \Omega.
		\]
		Then $u_1 \le u_2$ in $ \Omega$.
	\end{lemma}
	
	To obtain the sharp upper bound in the case $0<\gamma<1$, the following lemma is also needed.
	\begin{lemma}[Lemma 10 in \cite{2025arXiv250807360L}]\label{lem:U}
		Assume $0<\gamma<1$. Then for every $M>0$ and $r_2>r_1>0$, the problem
		\[
		\begin{cases}
			-\Delta_p u = \frac{M}{u^\gamma} &\text{ in } B_{r_2}\setminus\overline{B_{r_1}},\\
			u>0 &\text{ in } B_{r_2}\setminus\overline{B_{r_1}},\\
			u=0 &\text{ on } \partial(B_{r_2}\setminus\overline{B_{r_1}})
		\end{cases}
		\]
		admits a unique solution $U \in C^{1,\alpha}(B_{r_2}\setminus\overline{B_{r_1}}) \cap C(\overline{B_{r_2}}\setminus B_{r_1})$. Moreover, there exist \new{$c,C>0$} such that
		\[
		c\, {\rm dist}\big(x, \new{\partial(}B_{r_2}\setminus\overline{B_{r_1}}\new{)}\big) \le U(x) \le C\, {\rm dist}\big(x, \new{\partial(}B_{r_2}\setminus\overline{B_{r_1}}\new{)}\big) \quad\text{ for all } x \in B_{r_2}\setminus\overline{B_{r_1}}.
		\]
	\end{lemma}
	
	With the above lemmas, we are ready to prove Theorem \ref{th:boundary_bound} via comparison arguments.
	\begin{proof}[Proof of Theorem \ref{th:boundary_bound}]
		Since $u$ is bounded in $\Sigma_{\overline\lambda}$ and $\essinf_K u>0$, the function $\frac{1}{u^\gamma} + f(u)$ is bounded in each compact set $K \subset\subset\rev{\Sigma_{\overline\lambda}}$. \rev{As $0<q\le p$, the gradient term has at most natural growth, and} we have that $u\in C^{1,\alpha}_{\rm loc}(\rev{\Sigma_{\overline\lambda}})$ by standard elliptic regularity (see \cite{MR709038,MR727034}).
		
		Setting
		\begin{equation}\label{L}
			x_0 = -\overline\lambda e_N,\quad \Omega = B_{3\overline\lambda}(x_0)\setminus \overline{B_{\overline\lambda}(x_0)} \quad\text{ and } \quad L=\sup_{\Sigma_{\overline\lambda}} u.
		\end{equation}
		Let $\lambda_1>0$ be the first eigenvalue and $\phi_1 \in C^1(\overline{\Omega})$ a corresponding positive normalized eigenfunction of the $p$-Laplacian in $\Omega$ with the Dirichlet boundary condition, namely,
		\[
		\begin{cases}
			-\Delta_p \phi_1 = \lambda_1 \phi_1^{p-1} & \text{ in } \Omega,\\
			\phi_1 > 0 & \text{ in } \Omega,\\
			\phi_1 = 0 & \text{ on } \partial \Omega,
		\end{cases}
		\]
		and $\|\phi_1\|_{L^\infty(\Omega)}=1$.	
		We consider three cases.
		
		\textit{Case 1: $\gamma>1$.}		
		Let
		\[
		w_s = s \phi_1^\frac{p}{\gamma+p-1},
		\]
		where $s>0$. A direct calculation yields
		\[
		-\Delta_p w_s = \frac{a_s(x)}{w_s^\gamma} \quad\text{ in } \Omega,
		\]
		where
		\[
		a_s(x) = s^{\gamma+p-1} \left(\frac{p}{\gamma+p-1}\right)^{p-1} \left[\lambda_1\phi_1(x)^p + \frac{(\gamma-1)(p-1)}{\gamma+p-1}|\nabla\phi_1(x)|^p\right].
		\]
		
		Since $u \le L$ in $\Sigma_{\overline\lambda}$, we can find $M>0$ sufficiently large such that
		\begin{equation}\label{M}
			\frac{1}{u^\gamma} + f(u) < \frac{M}{u^\gamma} \quad\text{ for } x\in \Sigma_{\overline\lambda}.
		\end{equation}
		On the other hand, by the positivity of $\phi_1$ and the Hopf boundary lemma for $\phi_1$, we know that
		\[
		\inf_{x\in \Omega} \left[\lambda_1\phi_1(x)^p + \frac{(\gamma-1)(p-1)}{\gamma+p-1}|\nabla\phi_1(x)|^p\right] > 0.
		\]
		Hence, we can choose some large $s_1>0$ such that
		\[
		a_{s_1}(x) > M \text{ for } x\in \Omega \quad\text{ and }\quad w_{s_1}(x) > L \text{ for } x\in \partial B_{2\overline\lambda}(x_0).
		\]
		
		Let any
		\begin{equation}\label{y}
			y=(y',y_N)\in\Sigma_{\overline\lambda}
		\end{equation}
		and set $\hat{y}=(y',-\overline\lambda)$, we have
		\[
		\begin{cases}
			-\Delta_p u \le \frac{M}{u^\gamma} &\text{ in }  B_{2\overline\lambda}(\hat{y})\cap\Sigma_{\overline\lambda},\\
			-\Delta_p w_{s_1,\hat{y}} \ge \frac{M}{w_{s_1,\hat{y}}^\gamma} &\text{ in }  B_{2\overline\lambda}(\hat{y})\cap\Sigma_{\overline\lambda},\\
			w_{s_1,\hat{y}}, u > 0 &\text{ in }  B_{2\overline\lambda}(\hat{y})\cap\Sigma_{\overline\lambda},\\
			w_{s_1,\hat{y}} \ge u &\text{ on } \partial (B_{2\overline\lambda}(\hat{y})\cap\Sigma_{\overline\lambda}),
		\end{cases}
		\]
		where \new{$w_{s_1,\hat{y}}(x) = w_{s_1}(x-\hat{y}+x_0)$ (note that $\hat y-x_0=(y',0)$, so that $w_{s_1,\hat y}$ is well defined on the annulus $B_{3\overline\lambda}(\hat y)\setminus\overline{B_{\overline\lambda}(\hat y)}\supset B_{2\overline\lambda}(\hat{y})\cap\Sigma_{\overline\lambda}$)}.
		Lemma \ref{lem:wcp} (with $\vartheta=0$) can be applied to yield
		\[
		u \le w_{s_1,\hat{y}} \quad \text{ in } B_{2\overline\lambda}(\hat{y})\cap\Sigma_{\overline\lambda}.
		\]
		
		Since $y$ is arbitrary, we have
		\begin{equation}\label{upper_bound>1}
			u(x) \le C x_N^\frac{p}{\gamma+p-1} \text{ in }  \Sigma_{\overline\lambda} \quad\text{ for some } C>0.
		\end{equation}
		This also implies \rev{that $u(x)\to0$ as $x_N\to0^+$ uniformly in $x'$, so that $u$, extended by zero, is continuous on $\Sigma_{\overline\lambda}\cup\partial\mathbb{R}^N_+$}.
		
		Next, we prove the lower bound. 
		Exploiting \eqref{upper_bound>1}, we can find $\lambda_0<\frac{\overline\lambda}{2}$ such that
		\begin{equation}\label{lambda_0}
			\frac{1}{u^\gamma} + f(u) > \frac{1}{2u^\gamma} \quad\text{ for } x\in \Sigma_{2\lambda_0}.
		\end{equation}
		Setting
		\[
		\tilde x_0 = (0,\dots,0,\lambda_0).
		\]
		Let $\tilde\lambda_1>0$ and $\tilde\phi_1 \in C^1(\overline{B_{\lambda_0}(\tilde x_0)})$ be the first eigenvalue and a corresponding positive eigenfunction of the $p$-Laplacian in $B_{\lambda_0}(\tilde x_0)$ with the Dirichlet boundary condition, namely,
		\[
		\begin{cases}
			-\Delta_p \tilde\phi_1 = \tilde\lambda_1 \tilde\phi_1^{p-1} & \text{ in } B_{\lambda_0}(\tilde x_0),\\
			\tilde\phi_1 > 0 & \text{ in } B_{\lambda_0}(\tilde x_0),\\
			\tilde\phi_1 = 0 & \text{ on } \partial B_{\lambda_0}(\tilde x_0).
		\end{cases}
		\]
		
		Setting
		\[
		\tilde w_s = s \tilde\phi_1^\frac{p}{\gamma+p-1},
		\]
		where $s>0$. Then
		\[
		-\Delta_p \tilde w_s + \vartheta |\nabla \tilde w_s|^p + \vartheta = \frac{\tilde b_s(x)}{\tilde w_s^\gamma} \quad \text{ in } B_{\lambda_0}(\tilde x_0),
		\]
		where
		\begin{align*}
			\tilde b_s(x)
			&= s^{\gamma + p - 1} \left( \frac{p}{\gamma + p - 1} \right)^{p - 1} \tilde\lambda_1 \tilde\phi_1^p\\
			&\quad + s^{\gamma + p - 1} \left( \frac{p}{\gamma + p - 1} \right)^{p - 1} \left( \frac{\vartheta sp}{\gamma + p - 1} \new{\tilde\phi_1^\frac{p}{\gamma+p-1}} + \frac{(\gamma - 1)(p - 1)}{\gamma + p - 1} \right) |\nabla \tilde\phi_1|^p\\
			&\quad + \vartheta s^\gamma \tilde\phi_1^\frac{\gamma p}{\gamma+p-1}.
		\end{align*}
		
		We also choose some small $s_2>0$ such that
		\[
		\tilde b_{s_2}(x)<\frac12\quad\text{for }x\in B_{\lambda_0}(\tilde x_0),\qquad
		\text{and, if }\vartheta>0,\quad
		\tilde w_{s_2}(x)\le\left(\frac{1}{2\vartheta}\right)^{1/\gamma}.
		\]
		
		Let any
		\begin{equation}\label{y2}
			y=(y',y_N)\in\Sigma_{\lambda_0}
		\end{equation}
		and set $\hat{y}=(y',\lambda_0)$.
		Since $0<q\le p$, it follows that $|\nabla u|^q < |\nabla u|^p + 1$. Therefore, we have
		\[
		\begin{cases}
			-\Delta_p \tilde w_{s_2,\hat{y}} + \vartheta |\nabla \tilde w_{s_2,\hat{y}}|^p + \vartheta \le \frac{1}{2 \tilde w_{s_2,\hat{y}}^\gamma} &\text{ in }  B_{\lambda_0}(\hat{y}),\\
			-\Delta_p u + \vartheta |\nabla u|^p + \vartheta \ge \frac{1}{2 u^\gamma} &\text{ in }  B_{\lambda_0}(\hat{y}),\\
			\tilde w_{s_2,\hat{y}}, u > 0 &\text{ in }  B_{\lambda_0}(\hat{y}),\\
			\tilde w_{s_2,\hat{y}} \le u &\text{ on } \partial B_{\lambda_0}(\hat{y}),
		\end{cases}
		\]	
		where \new{$\tilde w_{s_2,\hat{y}}(x) = \tilde w_{s_2}(x-\hat{y}+\tilde x_0)$, which is well defined on $B_{\lambda_0}(\hat y)$ since $\hat y-\tilde x_0=(y',0)$}.
		Lemma \ref{lem:wcp} can be applied now to yield
		\[
		\tilde w_{s_2,\hat{y}}\le u \quad \text{ in }  B_{\lambda_0}(\hat{y}).
		\]
		This implies
		\[
		u(x) \ge c x_N^\frac{p}{\gamma+p-1} \text{ in }  \Sigma_{\lambda_0} \quad\text{ for some } c>0.
		\]
		
		\textit{Case 2: $\gamma=1$.} In this case, we set
		\[
		w_s = s \phi_1 (1- \ln \phi_1)^\frac{1}{p},
		\]
		where $s>0$. By direct calculations, we obtain
		\[
		-\Delta_p w_s = \frac{a_s(x)}{w_s} \quad\text{ in } \Omega,
		\]
		where
		\begin{align*}
			&a_s(x)\\
			&= s^p \left[z\left(z - \frac{1}{p} z^{1-p}\right)^{p-1} \lambda_1 \phi_1^p + \frac{p-1}{p} \left(1 - \frac{1}{p} z^{-p}\right)^{p-2} \left(1 + \frac{p-1}{p} z^{-p}\right) |\nabla\phi_1|^p\right]
		\end{align*}
		with $z:=(1- \ln \phi_1)^\frac{1}{p}$. The range of $z$ is $[1,+\infty)$. Moreover, by L'Hôpital's rule, we have $\phi_1 z^k \to 0$ as ${\rm dist}(x,\partial\Omega)\to0$ for any $k>0$. Hence, both terms in the square bracket are bounded above in $\Omega$. Their sum is also bounded below by a positive constant.
		
		
		Let $M>0$ as in \eqref{M} and choose some large $s_1>0$ such that
		\[
		a_{s_1}(x) > M \text{ for } x\in \Omega \quad\text{ and }\quad w_{s_1}(x) > L \text{ for } x\in \partial B_{2\overline\lambda}(x_0).
		\]
		
		Let any $y=(y',y_N)\in\Sigma_{\overline\lambda}$ and set $\hat{y}=(y',-\overline\lambda)$, we have
		\[
		\begin{cases}
			-\Delta_p u \le \frac{M}{u} &\text{ in }  B_{2\overline\lambda}(\hat{y})\cap\Sigma_{\overline\lambda},\\
			-\Delta_p w_{s_1,\hat{y}} \ge \frac{M}{w_{s_1,\hat{y}}} &\text{ in }  B_{2\overline\lambda}(\hat{y})\cap\Sigma_{\overline\lambda},\\
			w_{s_1,\hat{y}}, u > 0 &\text{ in }  B_{2\overline\lambda}(\hat{y})\cap\Sigma_{\overline\lambda},\\
			w_{s_1,\hat{y}} \ge u &\text{ on } \partial (B_{2\overline\lambda}(\hat{y})\cap\Sigma_{\overline\lambda}),
		\end{cases}
		\]
		where \new{$w_{s_1,\hat{y}}(x) = w_{s_1}(x-\hat{y}+x_0)$, as in Case 1}.
		Lemma \ref{lem:wcp} (with $\vartheta=0$) can be applied to yield
		\[
		u \le w_{s_1,\hat{y}} \quad \text{ in } B_{2\overline\lambda}(\hat{y})\cap\Sigma_{\overline\lambda}.
		\]
		
		Since $y$ is arbitrary, we have
		\[
		u(x) \le C x_N \left(1-\ln x_N\right)^\frac{1}{p} \text{ in }\Sigma_{\lambda_*},
		\quad\text{for some }0<\lambda_*<\min\{1,\overline\lambda\}\text{ and }C>0.
		\]
		This also implies \rev{that $u(x)\to0$ as $x_N\to0^+$ uniformly in $x'$, so that $u$, extended by zero, is continuous on $\Sigma_{\overline\lambda}\cup\partial\mathbb{R}^N_+$}.
		
		To derive the lower bound, we let $\lambda_0<\frac{\overline\lambda}{2}$ such that \eqref{lambda_0} holds and set	
		\[
		\tilde w_s = s \tilde \phi_1 (1 - \ln \tilde \phi_1)^\frac{1}{p},
		\]
		where $s>0$ and $\tilde\phi_1$ is as in Case 1. Then
		\[
		-\Delta_p \tilde w_s + \vartheta |\nabla \tilde w_s|^p + \vartheta = \frac{\tilde b_s(x)}{\tilde w_s} \quad\text{ in } B_{\lambda_0}(\tilde x_0),
		\]
		where
		\begin{align*}
			&\tilde b_s(x)\\
			&= s^p \left[\tilde z\left(\tilde z - \frac{1}{p} \tilde z^{1-p}\right)^{p-1} \tilde \lambda_1 \tilde \phi_1^p + \frac{p-1}{p} \left(1 - \frac{1}{p} \tilde z^{-p}\right)^{p-2} \left(1 + \frac{p-1}{p} \tilde z^{-p}\right) |\nabla\tilde \phi_1|^p\right]\\
			&\quad+ \new{\vartheta s^{p+1} \tilde \phi_1 \tilde z \left(\tilde z - \frac{1}{p} \tilde z^{1-p}\right)^p |\nabla \tilde \phi_1|^p + \vartheta s \tilde\phi_1 \tilde z}.
		\end{align*}
		
		We choose some small $s_0>0$ such that
		\[
		\tilde b_{s_0}(x)<\frac12\quad\text{for }x\in B_{\lambda_0}(\tilde x_0),\qquad
		\text{and, if }\vartheta>0,\quad
		\tilde w_{s_0}(x)\le\frac{1}{2\vartheta}.
		\]
		
		Let any $y=(y',y_N)\in\Sigma_{\lambda_0}$ and set $\hat{y}=(y',\lambda_0)$. Then we have
		\[
		\begin{cases}
			-\Delta_p \tilde w_{s_0,\hat{y}} + \vartheta |\nabla \tilde w_{s_0,\hat{y}}|^p + \vartheta \le \frac{1}{2 \tilde w_{s_0,\hat{y}}} &\text{ in }  B_{\lambda_0}(\hat{y}),\\
			-\Delta_p u + \vartheta |\nabla u|^p + \vartheta \ge \frac{1}{2 u} &\text{ in }  B_{\lambda_0}(\hat{y}),\\
			\tilde w_{s_0,\hat{y}}, u > 0 &\text{ in }  B_{\lambda_0}(\hat{y}),\\
			\tilde w_{s_0,\hat{y}} \le u &\text{ on } \partial B_{\lambda_0}(\hat{y}),
		\end{cases}
		\]	
		where \new{$\tilde w_{s_0,\hat{y}}(x) = \tilde w_{s_0}(x-\hat{y}+\tilde x_0)$}.
		Lemma \ref{lem:wcp} can be applied now to yield
		\[
		\tilde w_{s_0,\hat{y}}\le u \quad \text{ in }  B_{\lambda_0}(\hat{y}).
		\]
		This implies
		\[
		u(x) \ge c x_N \left(1-\ln x_N\right)^\frac{1}{p} \text{ in }  \Sigma_{\lambda_0} \quad\text{ for some } c>0.
		\]
		
		\textit{Case 3: $0<\gamma<1$.}	
		Let $M>0$ be as in \eqref{M}. Let $U$ be defined as in Lemma \ref{lem:U} with \new{$r_1=\overline\lambda$ and $r_2=3\overline\lambda$}. Then let $s>1$ be sufficiently large such that $sU>L$ on $\partial B_{2\overline\lambda}$.
		
		Let any $y=(y',y_N)\in\Sigma_{\overline\lambda}$ and set $\hat{y}=(y',-\overline\lambda)$, we have
		\[
		\begin{cases}
			-\Delta_p u \le \frac{M}{u^\gamma} &\text{ in }  B_{2\overline\lambda}(\hat{y})\cap\Sigma_{\overline\lambda},\\
			-\Delta_p U_{\hat{y}} \ge \frac{M s^{\gamma+p-1}}{U_{\hat{y}}^\gamma} \ge \frac{M}{U_{\hat{y}}^\gamma} &\text{ in }  B_{2\overline\lambda}(\hat{y})\cap\Sigma_{\overline\lambda},\\
			U_{\hat{y}}, u > 0 &\text{ in }  B_{2\overline\lambda}(\hat{y})\cap\Sigma_{\overline\lambda},\\
			U_{\hat{y}} \ge u &\text{ on } \partial (B_{2\overline\lambda}(\hat{y})\cap\Sigma_{\overline\lambda}),
		\end{cases}
		\]
		where $U_{\hat{y}}(x) = sU(x-\hat{y})$.
		Lemma \ref{lem:wcp} (with $\vartheta=0$) can be applied to yield
		\[
		u \le U_{\hat{y}} \quad \text{ in } B_{2\overline\lambda}(\hat{y})\cap\Sigma_{\overline\lambda}.
		\]
		
		Since $y$ is arbitrary, we can use \new{the upper bound in Lemma \ref{lem:U}} to deduce
		\[
		u(x) \le C x_N \text{ in }  \Sigma_{\overline\lambda} \quad\text{ for some } C>0.
		\]
		This also implies \rev{that $u$, extended by zero, is continuous on $\Sigma_{\overline\lambda}\cup\partial\mathbb{R}^N_+$}.
		
		To derive the lower bound, we let $\lambda_0<\frac{\overline\lambda}{2}$ such that \eqref{lambda_0} holds and set
		\[
		v_s = s \tilde\phi_1,
		\]
		where $s>0$ and $\tilde\phi_1$ is as in Case 1. We have
		\[
		-\Delta_p v_s + \vartheta |\nabla v_s|^p + \vartheta = \frac{b_s(x)}{v_s^\gamma} \quad\text{ in } B_{\lambda_0}(\tilde x_0),
		\]
		where
		\[
		b_s(x) = s^{p + \gamma - 1} \tilde\lambda_1 \tilde\phi_1^{p + \gamma - 1}
		+ \vartheta s^{p + \gamma} \tilde\phi_1^\gamma |\nabla \tilde\phi_1|^p
		+ \vartheta s^\gamma \tilde\phi_1^\gamma.
		\]
		
		We also choose some small $s_0>0$ such that
		\[
		b_{s_0}(x)<\frac12\quad\text{for }x\in B_{\lambda_0}(\tilde x_0),\qquad
		\text{and, if }\vartheta>0,\quad
		v_{s_0}(x)\le\left(\frac{1}{2\vartheta}\right)^{1/\gamma}.
		\]
		
		Let any $y=(y',y_N)\in\Sigma_{\lambda_0}$ and set $\hat{y}=(y',\lambda_0)$. Then we have
		\[
		\begin{cases}
			-\Delta_p v_{s_0,\hat{y}} + \vartheta |\nabla v_{s_0,\hat{y}}|^p + \vartheta \le \frac{1}{2 v_{s_0,\hat{y}}^\gamma} &\text{ in }  B_{\lambda_0}(\hat{y}),\\
			-\Delta_p u + \vartheta |\nabla u|^p + \vartheta \ge \frac{1}{2 u^\gamma} &\text{ in }  B_{\lambda_0}(\hat{y}),\\
			v_{s_0,\hat{y}}, u > 0 &\text{ in }  B_{\lambda_0}(\hat{y}),\\
			v_{s_0,\hat{y}} \le u &\text{ on } \partial B_{\lambda_0}(\hat{y}),
		\end{cases}
		\]	
		where \new{$v_{s_0,\hat{y}}(x) = v_{s_0}(x-\hat{y}+\tilde x_0)$}.
		Lemma \ref{lem:wcp} can be applied now to yield
		\[
		v_{s_0,\hat{y}}\le u \quad \text{ in }  B_{\lambda_0}(\hat{y}).
		\]
		This implies
		\[
		u(x) \ge c x_N \text{ in }  \Sigma_{\lambda_0} \quad\text{ for some } c>0.
		\]
		
		This completes the proof.
	\end{proof}
	
	\new{\subsection{The sharp constant in the borderline case}
	Theorem \ref{th:boundary_bound}(ii) provides two-sided bounds with unspecified constants. We now show that these constants are universal. The key ingredient is the following explicit family of radial barriers, whose leading term is the one-dimensional profile of the equation; all comparisons below are performed on \emph{bounded} domains.
	
	\begin{lemma}\label{lem:log_barrier}
		Let $\kappa>0$ and $K\in\mathbb{R}$. Set
		\[
		h(t):=\kappa\,t\left(K-\ln t\right)^\frac{1}{p}, \qquad z:=\left(K-\ln t\right)^\frac{1}{p} \qquad\text{ for } 0<t<e^{K}.
		\]
		Then
		\[
		h'(t)=\kappa\left(z-\tfrac{1}{p}z^{1-p}\right), \qquad
		h''(t)=-\frac{\kappa}{p\,t}\left(z^{1-p}+\tfrac{p-1}{p}z^{1-2p}\right),
		\]
		and, whenever $z^p>\frac{1}{p}$,
		\begin{equation}\label{barrier_1d}
			\begin{gathered}
				-(p-1)\left(h'(t)\right)^{p-2}h''(t) = \frac{(p-1)\kappa^p}{p}\,\Theta(z)\,\frac{1}{h(t)},\\
				\Theta(z):=\left(1-\tfrac{1}{p}z^{-p}\right)^{p-2}\left(1+\tfrac{p-1}{p}z^{-p}\right),
			\end{gathered}
		\end{equation}
		where $\Theta(z)\to1$ as $z\to+\infty$.
		Consequently, given $x_c\in\mathbb{R}^N$, $\rho>0$ and $\sigma\in\{-1,1\}$, the function
		\[
		V(x):=h\left(d(x)\right), \qquad d(x):=\sigma\left(|x-x_c|-\rho\right),
		\]
		is smooth and positive on $\left\{x\in\mathbb{R}^N \mid x\ne x_c,\ 0<d(x)<e^{K}\right\}$ and satisfies there, at every point where $z^p>\frac{1}{p}$ (with $z:=\left(K-\ln d(x)\right)^{1/p}$),
		\begin{align}
			\left(-\Delta_p V\right)V &= \frac{(p-1)\kappa^p}{p}\,\Theta(z) - \sigma\,\frac{N-1}{|x-x_c|}\,\kappa^p\left(1-\tfrac{1}{p}z^{-p}\right)^{p-1} d(x)\,z^{p}, \label{barrier_radial}\\
			\left|\nabla V\right|^p V &= \kappa^{p+1}\left(1-\tfrac{1}{p}z^{-p}\right)^{p} d(x)\,z^{p+1}. \label{barrier_grad}
		\end{align}
	\end{lemma}
	
	\begin{proof}
		The formulas for $h'$ and $h''$ follow from $\frac{dz}{dt}=-\frac{1}{p\,t}z^{1-p}$. If $z^p>\frac{1}{p}$ then $h'>0$ and
		\begin{align*}
			-(p-1)\left(h'\right)^{p-2}h''
			&= \frac{(p-1)\kappa^{p}}{p\,t}\,z^{p-2}\left(1-\tfrac{1}{p}z^{-p}\right)^{p-2}\,z^{1-p}\left(1+\tfrac{p-1}{p}z^{-p}\right)\\
			&= \frac{(p-1)\kappa^{p}}{p\,t\,z}\left(1-\tfrac{1}{p}z^{-p}\right)^{p-2}\left(1+\tfrac{p-1}{p}z^{-p}\right),
		\end{align*}
		which is \eqref{barrier_1d} because $h(t)=\kappa\,t\,z$.
		Next, write $r:=|x-x_c|$, so that $d=\sigma(r-\rho)$, $V=h(d)$, $\partial_r V=\sigma\,h'(d)$ and $\partial^2_{rr}V=h''(d)$. Since $\Delta_p V=(p-1)|\partial_rV|^{p-2}\partial^2_{rr}V+\frac{N-1}{r}|\partial_rV|^{p-2}\partial_rV$, we get
		\[
		-\Delta_p V = -(p-1)\left(h'(d)\right)^{p-2}h''(d) - \sigma\,\frac{N-1}{r}\left(h'(d)\right)^{p-1}.
		\]
		Multiplying by $V=\kappa\,d\,z$ and using $\left(h'(d)\right)^{p-1}=\kappa^{p-1}z^{p-1}\left(1-\frac{1}{p}z^{-p}\right)^{p-1}$ yields \eqref{barrier_radial}, while \eqref{barrier_grad} follows from $\left|\nabla V\right|=h'(d)$.
	\end{proof}
	
	\begin{proof}[Proof of Theorem \ref{th:sharp=1}]
		Set $\kappa_*:=\left(\frac{p}{p-1}\right)^{1/p}$ and fix $\varepsilon\in\left(0,\frac{1}{4}\right)$.
		By Theorem \ref{th:boundary_bound}(ii) there exist $0<\lambda_0<\min\left\{1,\overline\lambda\right\}$ and $c_2\ge c_1>0$ such that
		\begin{equation}\label{rough=1}
			c_1 x_N\left(1-\ln x_N\right)^\frac{1}{p} \le u(x) \le c_2 x_N\left(1-\ln x_N\right)^\frac{1}{p} \quad\text{ in } \Sigma_{\lambda_0}.
		\end{equation}
		In particular $u\to0$ uniformly as $x_N\to0^+$. Since $f$ is locally Lipschitz continuous on $[0,+\infty)$, it is bounded near $0$, so that $tf(t)\to0$ as $t\to0^+$ and there exists $\delta_1\in(0,\lambda_0)$ with
		\begin{equation}\label{Meps}
			\frac{1-\varepsilon}{u} \le \frac{1}{u}+f(u) \le \frac{1+\varepsilon}{u} \quad\text{ in } \Sigma_{\delta_1}.
		\end{equation}
		Recalling that $|\nabla u|^q\le|\nabla u|^p+1$ and $\vartheta|\nabla u|^q\ge0$, we deduce from \eqref{Meps} that
		\begin{equation}\label{ueq_pm}
			-\Delta_p u \le \frac{1+\varepsilon}{u},
			\qquad
			-\Delta_p u+\vartheta|\nabla u|^p+\vartheta \ge \frac{1-\varepsilon}{u}
			\quad\text{ in } \Sigma_{\delta_1}.
		\end{equation}
		Finally, by Lemma \ref{lem:log_barrier} we may fix $Z_\varepsilon>1$ so large that
		\begin{equation}\label{Zeps}
			1-\frac{\varepsilon}{4} \le \Theta(z) \le 1+\frac{\varepsilon}{4} \quad\text{ for all } z\ge Z_\varepsilon.
		\end{equation}
		Throughout the proof we shall use the following elementary facts: $0<\left(1-\frac{1}{p}z^{-p}\right)^{p-1}\le1$ and $0<\left(1-\frac{1}{p}z^{-p}\right)^{p}\le1$ whenever $z\ge Z_\varepsilon$; the map $t\mapsto t\left(1-\ln t\right)^{1/p}$ is increasing on $(0,1)$; and, for $s>0$, the map $t\mapsto t\left(K-\ln t\right)^{s}$ is increasing as long as $K-\ln t>s$.
		
		\textit{Step 1: the upper bound.} Let $\kappa^p:=\kappa_*^p(1+2\varepsilon)$, so that $\frac{(p-1)\kappa^p}{p}=1+2\varepsilon$. \rev{For $\delta\in(0,\delta_1]$ we put
		\[
		R_\delta:=\max\left\{\Big(\frac{c_2}{\kappa}\Big)^p\left(1-\ln\delta\right),\,Z_\varepsilon^p,\,2\right\},\qquad K_\delta:=R_\delta+\ln\delta,\qquad \rho_\delta:=\frac{2(N-1)\kappa^p\,\delta\,R_\delta}{\varepsilon}.
		\]
		Since $\delta^2R_\delta\to0$ as $\delta\to0^+$, we can fix $\delta\in(0,\delta_1]$ such that
		\begin{equation}\label{delta_small}
			2\rho_\delta\delta+\delta^2\le1,
		\end{equation}
		and from now on we write $R:=R_\delta$, $K:=K_\delta$ and $\rho:=\rho_\delta$. Thus $R=K-\ln\delta\ge\max\{Z_\varepsilon^p,2\}$,
		\begin{equation}\label{outer_up}
			\kappa^pR\ge c_2^p\left(1-\ln\delta\right)
		\end{equation}
		and
		\begin{equation}\label{rho_large}
			\frac{(N-1)\kappa^p\,\delta\,R}{\rho} = \frac{\varepsilon}{2}.
		\end{equation}}
		
		Fix $y'\in\mathbb{R}^{N-1}$ and apply Lemma \ref{lem:log_barrier} with $x_c:=(y',-\rho)$ and $\sigma:=1$, so that $d(x)=|x-x_c|-\rho>0$ on the whole of $\mathbb{R}^N_+$. Consider the open set
		\[
		\Omega:=\left\{x\in\Sigma_\delta \mid d(x)<\delta\right\},
		\]
		which is nonempty and \emph{bounded}, since $d(x)<\delta$ forces
		\[
		|x'-y'|^2 < (\rho+\delta)^2-(x_N+\rho)^2 \le 2\rho\delta+\delta^2\rev{\le1}.
		\] Moreover, if $x_N=\delta$ then $d(x)\ge(\delta+\rho)-\rho=\delta$, whence
		\[
		\partial\Omega \subset \left(\left\{x_N=0\right\}\cap\overline\Omega\right) \cup \left(\left\{d=\delta\right\}\cap\overline\Omega\right).
		\]
		In $\Omega$ we have $0<d<\delta$, hence $z^p=K-\ln d>K-\ln\delta=R\ge Z_\varepsilon^p$, and $d\,z^p=d\left(K-\ln d\right)\le\delta R$ because $R\ge2>1$. Since $|x-x_c|\ge\rho$, formulas \eqref{barrier_radial}, \eqref{Zeps} and \eqref{rho_large} give
		\[
		\left(-\Delta_p V\right)V \ge (1+2\varepsilon)\left(1-\frac{\varepsilon}{4}\right) - \frac{(N-1)\kappa^p\delta R}{\rho}
		\ge 1+\frac{3\varepsilon}{2}-\frac{\varepsilon}{2} = 1+\varepsilon
		\quad\text{ in } \Omega,
		\]
		that is, $-\Delta_p V\ge\frac{1+\varepsilon}{V}$ in $\Omega$.
		On $\partial\Omega$ we have either $x_N=0$, where $u=0\le V$, or $d=\delta$, where \rev{$0\le x_N\le\delta<\lambda_0$, so that \eqref{rough=1}, the monotonicity of $t\mapsto t(1-\ln t)^{1/p}$ on $(0,1)$ and \eqref{outer_up} give
			\[
			u(x)\le c_2\,\delta\left(1-\ln\delta\right)^\frac{1}{p}\le \kappa\,\delta\,R^\frac{1}{p}=V(x).
			\]} Since $\Omega\subset\Sigma_{\delta_1}$ is bounded and $u,V\in C^{1,\alpha}_{\rm loc}(\Omega)\cap L^\infty(\Omega)$, Lemma \ref{lem:wcp} (applied with $\vartheta=0$ and $c=1+\varepsilon$) together with the first inequality in \eqref{ueq_pm} yields $u\le V$ in $\Omega$.
		Evaluating at $x=(y',t)$ with $0<t<\delta$, for which $d(x)=t$, and recalling that $y'\in\mathbb{R}^{N-1}$ was arbitrary, we obtain
		\[
		u(y',t) \le \kappa\,t\left(K-\ln t\right)^\frac{1}{p} \qquad\text{ for all } y'\in\mathbb{R}^{N-1} \text{ and } 0<t<\delta.
		\]
		Since $\frac{K-\ln t}{1-\ln t}\to1$ as $t\to0^+$, we conclude
		\begin{equation}\label{limsup=1}
			\limsup_{t\to0^+}\;\sup_{y'\in\mathbb{R}^{N-1}} \frac{u(y',t)}{t\left(1-\ln t\right)^\frac{1}{p}} \le \kappa = \kappa_*\left(1+2\varepsilon\right)^\frac{1}{p}.
		\end{equation}
		
		\textit{Step 2: the lower bound.} Let now $\kappa^p:=\kappa_*^p(1-2\varepsilon)$, so that $\frac{(p-1)\kappa^p}{p}=1-2\varepsilon$, and set
		\[
		R:=\max\left\{Z_\varepsilon^p,2\right\}, \qquad \rho:=\frac{\delta_1}{4}.
		\]
		Note that $R\ge2>\frac{p+1}{p}$ since $p>1$. For $K<0$ we put $\delta_K:=e^{K-R}$, so that
		\[
		K-\ln\delta_K=R \qquad\text{ and }\qquad 1-\ln\delta_K=1+R-K.
		\]
		Since $\delta_K\to0$ and $-K\to+\infty$ as $K\to-\infty$, while $R$, $\kappa$, $\rho$ and $c_1$ are already fixed, we may choose $K<0$ so negative that
		\begin{equation}\label{K_choice}
			\begin{gathered}
				2\delta_K\le\min\left\{\delta_1,\rho\right\}, \qquad
			\text{if }\vartheta>0:\quad \kappa\,\delta_K R^\frac{1}{p}\le\frac{1-\varepsilon}{\vartheta},\\
				\frac{2(N-1)\kappa^p\delta_K R}{\rho}+\vartheta\kappa^{p+1}\delta_K R^\frac{p+1}{p}+\vartheta\kappa\,\delta_K R^\frac{1}{p} \le \frac{3\varepsilon}{4}
			\end{gathered}
		\end{equation}
		(the requirement containing $1/\vartheta$ is omitted when $\vartheta=0$) and, in addition,
		\begin{equation}\label{outer_bc}
			\kappa^p R \le c_1^p\left(1+R-K\right).
		\end{equation}
		
		Fix $y'\in\mathbb{R}^{N-1}$ and apply Lemma \ref{lem:log_barrier} with $x_c:=(y',\rho)$ and $\sigma:=-1$, so that $d(x)=\rho-|x-x_c|$ is positive exactly inside the ball $B_\rho(x_c)\subset\mathbb{R}^N_+$, which is internally tangent to $\partial\mathbb{R}^N_+$ at $(y',0)$. Consider the bounded open set
		\[
		\Omega:=\left\{x\in\mathbb{R}^N \mid 0<d(x)<\delta_K\right\}\cap\Sigma_{2\delta_K} \subset \Sigma_{\delta_1},
		\]
		whose boundary is contained in $\left\{d=0\right\}\cup\left\{d=\delta_K\right\}\cup\left\{x_N=2\delta_K\right\}$.
		In $\Omega$ we have $z^p=K-\ln d>K-\ln\delta_K=R\ge Z_\varepsilon^p$, $|x-x_c|\ge\rho-\delta_K\ge\frac{\rho}{2}$, and, by the monotonicity recalled above,
		\[
		d\,z^p\le\delta_K R, \qquad d\,z^{p+1}\le\delta_K R^\frac{p+1}{p}, \qquad V=\kappa\,d\,z\le\kappa\,\delta_K R^\frac{1}{p}.
		\]
		Hence \eqref{barrier_radial}, \eqref{barrier_grad}, \eqref{Zeps} and \eqref{K_choice} give
		\[
		\begin{aligned}
			\left(-\Delta_p V\right)V + \vartheta\left|\nabla V\right|^pV + \vartheta V
			&\le (1-2\varepsilon)\left(1+\frac{\varepsilon}{4}\right) + \frac{3\varepsilon}{4}\\
			&\le 1-\frac{7\varepsilon}{4}+\frac{3\varepsilon}{4} = 1-\varepsilon,
		\end{aligned}
		\]
		that is, $-\Delta_p V+\vartheta\left|\nabla V\right|^p+\vartheta\le\frac{1-\varepsilon}{V}$ in $\Omega$; moreover, if $\vartheta>0$, then $V\le\frac{1-\varepsilon}{\vartheta}$ by \eqref{K_choice}.
		
		We now check that $V\le u$ on $\partial\Omega$. If $d(x)=0$, then $V(x)=0\le u(x)$. If $d(x)=\delta_K$, then $x$ belongs to the sphere $\partial B_{\rho-\delta_K}(x_c)$, whose lowest point has height $\rho-(\rho-\delta_K)=\delta_K$; therefore $\delta_K\le x_N<2\delta_K<\lambda_0$ and, by \eqref{rough=1} and \eqref{outer_bc},
		\[
		u(x)\ge c_1\delta_K\left(1-\ln\delta_K\right)^\frac{1}{p} = c_1\delta_K\left(1+R-K\right)^\frac{1}{p}
		\ge \kappa\,\delta_K R^\frac{1}{p} \ge V(x).
		\]
		If $x_N=2\delta_K$, then similarly
		\[
		u(x)\ge c_1\,2\delta_K\left(1-\ln(2\delta_K)\right)^\frac{1}{p} \ge c_1\delta_K\left(1-\ln\delta_K\right)^\frac{1}{p}
		\ge \kappa\,\delta_K R^\frac{1}{p} \ge V(x).
		\]
		Consequently, Lemma \ref{lem:wcp} (applied with $c=1-\varepsilon$ on the bounded domain $\Omega$) and the second inequality in \eqref{ueq_pm} yield $V\le u$ in $\Omega$.
		Evaluating at $x=(y',t)$ with $0<t<\delta_K$, for which $d(x)=t$ and $x_N=t<2\delta_K$, and recalling that $y'$ is arbitrary, we get
		\[
		u(y',t) \ge \kappa\,t\left(K-\ln t\right)^\frac{1}{p} \qquad\text{ for all } y'\in\mathbb{R}^{N-1} \text{ and } 0<t<\delta_K,
		\]
		whence
		\[
		\liminf_{t\to0^+}\;\inf_{y'\in\mathbb{R}^{N-1}} \frac{u(y',t)}{t\left(1-\ln t\right)^\frac{1}{p}} \ge \kappa = \kappa_*\left(1-2\varepsilon\right)^\frac{1}{p}.
		\]
		
		Combining this with \eqref{limsup=1} and letting $\varepsilon\to0^+$ completes the proof.
		
		\rev{We note for later use that, for a fixed $y'$, the comparisons in Steps 1 and 2 take place in bounded sets contained in $\{x \mid |x'-y'|<1,\ 0<x_N<\delta_1\}$, by \eqref{delta_small} and since $\rho=\frac{\delta_1}{4}<1$ in Step 2. Hence the argument only uses \eqref{rough=1} and \eqref{Meps} in this cylinder.}
	\end{proof}}
	
	\subsection{Estimates for directional derivatives of solutions near the boundary}
	Once the sharp bounds near the boundary for solutions are available, we can derive Theorem \ref{th:boundary_behavior} via a scaling argument.
	\begin{proof}[Proof of Theorem \ref{th:boundary_behavior}]
		By Theorem \ref{th:boundary_bound}, $u\in C^{1,\alpha}_{\rm loc}(\rev{\Sigma_{\overline\lambda}})$ for some $0<\alpha<1$ \rev{and $u$ is continuous up to $\partial\mathbb{R}^N_+$}.
		We consider three cases.
		
		\textit{Case 1: $\gamma>1$.} In this case, there exist $\lambda_0,c_1,c_2>0$ such that
		\begin{equation}\label{u_bound>1}
			c_1 x_N^\frac{p}{\gamma+p-1} \le u(x) \le c_2 x_N^\frac{p}{\gamma+p-1} \quad\text{ in } \Sigma_{\lambda_0}.
		\end{equation}
		
		Let any sequences $x_n = (x_n', x_{n,N}) \in \mathbb{R}^N_+$ and $\eta_n \in \mathbb{S}^{N-1}$ such that
		\[
		\varepsilon_n:=x_{n,N}\to 0 \quad\text{ and }\quad \eta_n\to \eta_0 \quad\text{ as } n\to\infty.
		\]
		We define
		\[
		w_n(x) := \varepsilon_n^{-\frac{p}{\gamma+p-1}} u(\varepsilon_n x'+x_n', \varepsilon_n x_N) \quad\text{ for } x=(x',x_N) \in \mathbb{R}^N_+.
		\]
		
		Let $A>a>0$. For $n$ sufficiently large such that $\varepsilon_n<\frac{\lambda_0}{A}$, we deduce from \eqref{u_bound>1}
		\begin{equation}\label{wn_bound>1}
			c_1 a^\frac{p}{\gamma+p-1} \le c_1 x_N^\frac{p}{\gamma+p-1} \le w_n(x) \le c_2 x_N^\frac{p}{\gamma+p-1} \le c_2 A^\frac{p}{\gamma+p-1} \quad\text{ in } \Sigma_A\setminus \Sigma_a.
		\end{equation}
		Furthermore, $w_n$ solves
		\begin{equation}\label{wn_eq>1}
			-\Delta_p w_n + \varepsilon_n^{\frac{\gamma (p - q) + q}{\gamma + p - 1}} \vartheta |\nabla w_n|^q = \frac{1}{w_n^\gamma} + \varepsilon_n^{\frac{p \gamma}{\gamma + p - 1}} f\left( \varepsilon_n^{\frac{p}{\gamma + p - 1}} w_n \right) \quad \text{ in } \mathbb{R}^N_+.
		\end{equation}
		
		From \eqref{wn_bound>1} and the standard regularity \rev{\cite{MR709038,MR727034} (the coefficient of the gradient term in \eqref{wn_eq>1} being bounded)}, we know that $(w_n)$ is uniformly bounded in $C^{1,\alpha}(\overline{\Sigma_{\rev{A'}}\setminus \Sigma_{\rev{a'}}})$ \rev{for all $a<a'<A'<A$}. Hence, up to a subsequence, we have
		\[
		w_n \to w_{a,A} \quad\text{ in } C^{1,\alpha'}_{\rm loc}(\overline{\Sigma_A\setminus \Sigma_a}),
		\]
		where $0<\alpha'<\alpha<1$.
		Moreover, passing \eqref{wn_eq>1} to the limit, we get
		\[
		\new{-\Delta_p} w_{a,A} = \frac{1}{w_{a,A}^\gamma} \quad\text{ in } \Sigma_A\setminus \Sigma_a.
		\]
		Now we take $a = \frac{1}{j}$ and $A = j$, for large $j \in \mathbb{N}$ and determine $w_{\frac{1}{j},j}$ as above. For $j\to\infty$, using a standard diagonal process, we can construct a limiting profile $w_\infty \in C^{1,\alpha'}_{\rm loc}(\mathbb{R}^N_+)$ so that
		\begin{equation}\label{w_infty}
			\new{-\Delta_p} w_\infty = \frac{1}{w_\infty^\gamma} \quad\text{ in } \mathbb{R}^N_+
		\end{equation}
		and $w_{\frac{1}{j},j} = w_\infty$ in $\Sigma_j\setminus \Sigma_{\frac{1}{j}}$. Moreover, from \eqref{wn_bound>1} we know that
		\begin{equation}\label{w_infty_bounds}
			c_1 x_N^\frac{p}{\gamma+p-1} \le w_\infty(x) \le c_2 x_N^\frac{p}{\gamma+p-1} \quad\text{ in } \mathbb{R}^N_+.
		\end{equation}
		Hence by defining $w_\infty=0$ on $\partial\mathbb{R}^N_+$, we have $w_\infty \in C^{1,\alpha'}_{\rm loc}(\mathbb{R}^N_+)\cap C(\overline{\mathbb{R}^N_+})$ and $w_\infty$ is a solution to \eqref{main} with $\vartheta=0$ and $f\equiv0$.
		By \cite[Theorem 1.2]{MR4044739}, we have
		\[
		w_\infty(x) \equiv \left[\frac{(\gamma+p-1)^p}{p^{p-1}(p-1)(\gamma-1)}\right]^\frac{1}{\gamma+p-1}	x_N^\frac{p}{\gamma+p-1}.
		\]
		Therefore,
		\[
		x_{n,N}^\frac{\gamma-1}{\new{\gamma+p-1}} \frac{\partial u(x_n)}{\partial \eta_n} = \frac{\partial w_n(e_N)}{\partial \eta_n} \to \frac{\partial w_\infty(e_N)}{\partial \eta_0}=w_\infty'(1)\langle\eta_0,e_N\rangle.
		\]
		This proves \eqref{limit>1}. Then \eqref{derivative_bound>1} follows by a contradiction argument.
		
		Lastly, from \eqref{u_bound>1} and \new{\eqref{full_gradient}} we have
		\[
		\begin{aligned}
			\left|\nabla\left(u^\frac{\gamma+p-1}{p}\right)\right| &= \frac{\gamma+p-1}{p} u^\frac{\gamma-1}{p} |\nabla u|\\
			&\new{\le C\,\frac{\gamma+p-1}{p} c_2^\frac{\gamma-1}{p} x_N^\frac{\gamma-1}{\gamma+p-1} x_N^\frac{1-\gamma}{\gamma+p-1}}\\
			&\new{= C\,\frac{\gamma+p-1}{p} c_2^\frac{\gamma-1}{p}} \quad\text{ in } \Sigma_{\lambda_0}.
		\end{aligned}
		\]
		Hence $u^\frac{\gamma+p-1}{p}$ is Lipschitz continuous in \new{$\overline{\Sigma_{\lambda_0}}$}\rev{, the strip being convex}.
		\rev{Since $|a-b|\le|a^{1/\theta}-b^{1/\theta}|^\theta$ for $a,b\ge0$ and $\theta\in(0,1]$, this gives $u \in C^{0,\frac{p}{\gamma+p-1}}(\overline{\Sigma_{\lambda_0}})$.}
		
		\textit{Case 2: $\gamma=1$.} By Theorem \ref{th:boundary_bound}, there exist $0<\lambda_0<1$, $c_1,c_2>0$ such that
		\begin{equation}\label{u_bound=1}
			c_1 x_N \left(1-\ln x_N\right)^\frac{1}{p} \le u(x) \le c_2 x_N \left(1-\ln x_N\right)^\frac{1}{p} \quad\text{ in } \Sigma_{\lambda_0}.
		\end{equation}
		
		We prove the lower bound of \eqref{derivative_bound=1}. Suppose by contradiction that there exist $\beta>0$, a sequence of points $x_n\in\mathbb{R}^N_+$ and a sequence of vectors $\eta_n\in \mathbb{S}^{N-1}$ with $\langle\eta_n,e_N\rangle \ge \beta$ such that
		\begin{equation}\label{contra_assump2=1}
			\varepsilon_n:=x_{n,N}\to 0 \quad\text{ and }\quad \left(1-\ln \varepsilon_n\right)^{-\frac{1}{p}} \frac{\partial u(x_n)}{\partial \eta_n} \to 0 \quad\text{ as } n\to\infty.
		\end{equation}
		Up to a subsequence, we may assume
		\[
		\eta_n\to \eta_0 \quad\text{ as } n\to\infty
		\]
		with $\langle\eta_0, e_N\rangle \ge \beta$.
		We define
		\[
		w_n(x) := \varepsilon_n^{-1} \left(1-\ln \varepsilon_n\right)^{-\frac{1}{p}} u(\varepsilon_n x'+x_n', \varepsilon_n x_N) \quad\text{ for } x=(x',x_N) \in \mathbb{R}^N_+.
		\]
		
		For $A>a>0$ and $n$ sufficiently large, we deduce from \eqref{u_bound=1}
		\[
		c_1 x_N \left(\frac{1-\ln(\varepsilon_n x_N)}{1-\ln\varepsilon_n}\right)^\frac{1}{p} \le w_n(x) \le c_2 x_N \left(\frac{1-\ln(\varepsilon_n x_N)}{1-\ln\varepsilon_n}\right)^\frac{1}{p} \quad\text{ in } \Sigma_A\setminus \Sigma_a.
		\]
		Notice that
		\[
		\frac{1-\ln(A \varepsilon_n)}{1-\ln\varepsilon_n} \le \frac{1-\ln(\varepsilon_n x_N)}{1-\ln\varepsilon_n} \le \frac{1-\ln(a \varepsilon_n)}{1-\ln\varepsilon_n} \quad\text{ in } \Sigma_A\setminus \Sigma_a.
		\]
		Hence, for large $n$,
		\begin{equation}\label{wn_bound=1}
			\frac{c_1}{2} a \le \frac{c_1}{2} x_N \le w_n(x) \le 2 c_2 x_N \le 2 c_2 A \quad\text{ in } \Sigma_A\setminus \Sigma_a.
		\end{equation}
		
		Moreover,
		\begin{equation}\label{wn_eq=1}
			\begin{aligned}
				&-\Delta_p w_n + \varepsilon_n \left(1-\ln \varepsilon_n\right)^{\frac{q-p+1}{p}} \vartheta |\nabla w_n|^q\\
				&= \frac{\left(1-\ln \varepsilon_n\right)^{-1}}{w_n} + \varepsilon_n \left(1-\ln \varepsilon_n\right)^{-\frac{p-1}{p}} f\left(\varepsilon_n \left(1-\ln \varepsilon_n\right)^{\frac{1}{p}} w_n\right) \quad \text{ in } \mathbb{R}^N_+.
			\end{aligned}
		\end{equation}
		
		As in Case 1, $(w_n)$ is uniformly bounded in $C^{1,\alpha}(\overline{\Sigma_A\setminus \Sigma_a})$. Via the compactness embedding and a diagonal process, we have $w_n \to w_\infty$ in $C^{1,\alpha'}_{\rm loc}(\mathbb{R}^N_+)$ up to a subsequence for some $w_\infty \in C^{1,\alpha'}_{\rm loc}(\mathbb{R}^N_+)$.
		Moreover, \eqref{wn_bound=1} implies
		\[
		c_1 x_N \le w_\infty(x) \le c_2 x_N \quad\text{ in } \mathbb{R}^N_+.
		\]
		Hence by defining $w_\infty=0$ on $\partial\mathbb{R}^N_+$, we have $w_\infty \in C^{1,\alpha'}_{\rm loc}(\mathbb{R}^N_+)\cap C(\overline{\mathbb{R}^N_+})$.
		Passing $n\to\infty$ in \eqref{wn_eq=1}, we deduce
		\[
		\begin{cases}
			\new{-\Delta_p w_\infty} = 0 &\text{ in } \mathbb{R}^N_+,\\
			\new{w_\infty} > 0 &\text{ in } \mathbb{R}^N_+,\\
			\new{w_\infty} = 0 &\text{ on } \partial\mathbb{R}^N_+.
		\end{cases}
		\]
		By \cite[Theorem 3.1]{MR2337497}, we have
		\[
		w_\infty(x) \equiv a x_N,
		\]
		for some $a>0$.
		Therefore,
		\[
		\left(1-\ln \varepsilon_n\right)^{-\frac{1}{p}} \frac{\partial u(x_n)}{\partial \eta_n} = \frac{\partial w_n(e_N)}{\partial \eta_n} \to \frac{\partial w_\infty(e_N)}{\partial \eta_0}=\new{a}\langle\eta_0,e_N\rangle \ge \new{a}\beta>0.
		\]
		However, this contradicts \eqref{contra_assump2=1}. Hence, the lower bound of \eqref{derivative_bound=1} is proved. The upper bound can be obtained in the same way.
		
		\new{We now identify the constant $a$ and prove \eqref{limit=1}. Let $x_n\in\mathbb{R}^N_+$ and $\eta_n\in\mathbb{S}^{N-1}$ be \emph{arbitrary} sequences with $\varepsilon_n:=x_{n,N}\to0$ and $\eta_n\to\eta_0$, and let $w_n$, $w_\infty$ be as above. For every fixed $x=(x',x_N)\in\mathbb{R}^N_+$ we can write
		\[
		w_n(x)= \frac{u\left(\varepsilon_n x'+x_n',\varepsilon_n x_N\right)}{\varepsilon_n x_N\left(1-\ln(\varepsilon_n x_N)\right)^\frac{1}{p}}\; x_N \left(\frac{1-\ln(\varepsilon_n x_N)}{1-\ln\varepsilon_n}\right)^\frac{1}{p}.
		\]
		The first factor tends to $\kappa_*:=\left(\frac{p}{p-1}\right)^{1/p}$ by Theorem \ref{th:sharp=1} (the convergence there being uniform in the horizontal variable), while the second factor tends to $x_N$ because $\frac{1-\ln(\varepsilon_n x_N)}{1-\ln\varepsilon_n}\to1$ for $x_N$ fixed. Hence $w_\infty(x)=\kappa_* x_N$, that is, $a=\kappa_*$. Since $w_n\to w_\infty$ in $C^{1,\alpha'}_{\rm loc}(\mathbb{R}^N_+)$ and $\eta_n\to\eta_0$, we conclude
		\[
		\left(1-\ln \varepsilon_n\right)^{-\frac{1}{p}} \frac{\partial u(x_n)}{\partial \eta_n} = \frac{\partial w_n(e_N)}{\partial \eta_n} \to \frac{\partial w_\infty(e_N)}{\partial \eta_0}= \kappa_*\langle\eta_0,e_N\rangle,
		\]
		which is exactly \eqref{limit=1}.}
		
		Lastly, from \eqref{u_bound=1} and \rev{\eqref{full_gradient}}, we deduce for any $s\in(0,1)$,
		\[
		\left|\nabla\left(u^\frac{1}{s}\right)\right| = \new{\frac{1}{s}} u^{\frac{1}{s}-1} |\nabla u| \le C' x_N^{\frac{1}{s}-1} \left(1-\ln x_N\right)^\frac{1}{ps} < C'' \quad\text{ in } \Sigma_{\lambda_0},
		\]
		thanks to $\lim_{t\to0^+} t^{\frac{1}{s}-1} \left(\new{1}-\ln t\right)^\frac{1}{ps} = 0$.
		Hence \new{$u^\frac{1}{s}$} is Lipschitz continuous in \new{$\overline{\Sigma_{\lambda_0}}$}.
		This indicates $u \in \rev{C^{0,s}(\overline{\Sigma_{\lambda_0}})}$.
		
		\textit{Case 3: $0<\gamma<1$.} In this case, there exist $\lambda_0,c_1,c_2>0$ such that
		\begin{equation}\label{u_bound<1}
			c_1 x_N \le u(x) \le c_2 x_N \quad\text{ in } \Sigma_{\lambda_0}.
		\end{equation}
		
		We prove the lower bound of \eqref{derivative_bound<1}. Suppose by contradiction that there exist $\beta>0$, a sequence of points $x_n\in\mathbb{R}^N_+$ and a sequence of \new{unit} vectors $\eta_n\in \mathbb{S}^{N-1}$ with $\langle\eta_n,e_N\rangle \ge \beta$ such that
		\begin{equation}\label{contra_assump2<1}
			\varepsilon_n:=x_{n,N}\to 0 \quad\text{ and }\quad \frac{\partial u(x_n)}{\partial \eta_n} \to 0 \quad\text{ as } n\to\infty.
		\end{equation}
		Up to a subsequence, we may assume
		\[
		\eta_n\to \eta_0 \quad\text{ as } n\to\infty
		\]
		with $\langle\eta_0, e_N\rangle \ge \beta$.
		We define
		\[
		w_n(x) := \varepsilon_n^{-1} u(\varepsilon_n x'+x_n', \varepsilon_n x_N) \quad\text{ for } x=(x',x_N) \in \mathbb{R}^N_+.
		\]
		
		For $A>a>0$ and $n$ sufficiently large, we deduce from \eqref{u_bound<1}
		\begin{equation}\label{wn_bound<1}
			c_1 a \le c_1 x_N \le w_n(x) \le c_2 x_N \le c_2 A \quad\text{ in } \Sigma_A\setminus \Sigma_a.
		\end{equation}
		
		Moreover,
		\begin{equation}\label{wn_eq<1}
			-\Delta_p w_n + \varepsilon_n \vartheta |\nabla w_n|^q = \frac{\varepsilon_n^{1-\gamma}}{w_n^\gamma} + \varepsilon_n f\left( \varepsilon_n w_n \right) \quad \text{ in } \mathbb{R}^N_+.
		\end{equation}
		
		As in Case 1, $w_n \to w_\infty$ in $C^{1,\alpha'}_{\rm loc}(\mathbb{R}^N_+)$ up to a subsequence for some $w_\infty \in C^{1,\alpha'}_{\rm loc}(\mathbb{R}^N_+)$.
		Moreover, \eqref{wn_bound<1} implies
		\[
		c_1 x_N \le w_\infty(x) \le c_2 x_N \quad\text{ in } \mathbb{R}^N_+.
		\]
		Hence by defining $w_\infty=0$ on $\partial\mathbb{R}^N_+$, we have $w_\infty \in C^{1,\alpha'}_{\rm loc}(\mathbb{R}^N_+)\cap C(\overline{\mathbb{R}^N_+})$.
		Passing $n\to\infty$ in \eqref{wn_eq<1}, we deduce
		\[
		\begin{cases}
			\new{-\Delta_p w_\infty} = 0 &\text{ in } \mathbb{R}^N_+,\\
			\new{w_\infty} > 0 &\text{ in } \mathbb{R}^N_+,\\
			\new{w_\infty} = 0 &\text{ on } \partial\mathbb{R}^N_+.
		\end{cases}
		\]
		By \cite[Theorem 3.1]{MR2337497}, we have
		\[
		w_\infty(x) \equiv a x_N,
		\]
		for some $a>0$.
		Therefore,
		\[
		\frac{\partial u(x_n)}{\partial \eta_n} = \frac{\partial w_n(e_N)}{\partial \eta_n} \to \frac{\partial w_\infty(e_N)}{\partial \eta_0}=\new{a}\langle\eta_0,e_N\rangle \ge \new{a}\beta>0.
		\]
		However, this contradicts \eqref{contra_assump2<1}. Hence, the lower bound of \new{\eqref{derivative_bound<1}} is proved. The upper bound can be obtained in the same way.
		
		On the other hand, by \eqref{u_bound<1},
		\[
		\rev{\left|-\Delta_p u\right| = \left|\frac{1}{u^\gamma} + f(u) - \vartheta |\nabla u|^q\right| \le \frac{M}{u^\gamma} \le \frac{M}{c_1^\gamma x_N^\gamma}} \quad\text{ in } \Sigma_{\lambda_0}
		\]
		for some large $M>0$\rev{, where we used that $f$ is bounded on $[0,\sup_{\Sigma_{\lambda_0}}u]$ and that $|\nabla u|$ is bounded in $\Sigma_{\lambda_0}$ by \eqref{full_gradient}}.	
		Hence we can apply \cite[Theorem B.1]{MR2341518} to get $u\in C^{1,\alpha}_{\rm loc}(\rev{\Sigma_{\overline\lambda}\cup\partial\mathbb{R}^N_+})$ for some $0<\alpha<1$.
		
		This completes the proof.
	\end{proof}
	
	\begin{remark}\label{rem:loglip}
		\new{In the borderline case $\gamma=1$, the conclusion \rev{$u\in C^{0,s}(\overline{\Sigma_{\lambda_0}})$} for every $s\in(0,1)$ can be upgraded to the sharp \emph{log-Lipschitz} estimate: \rev{there are $\lambda_1\in(0,1)$ and $C>0$} such that
		\begin{equation}\label{loglip}
			|u(x)-u(y)| \le C\,|x-y|\left(1+\left|\ln|x-y|\right|\right)^\frac{1}{p} \quad \text{ for all } x,y \in \rev{\overline{\Sigma_{\lambda_1}}} \text{ with } 0<|x-y|<\rev{\lambda_1}.
		\end{equation}
		Indeed, \rev{let $\lambda_0<1$ be such that \eqref{full_gradient} holds in $\Sigma_{\lambda_0}$ and let $\lambda_1:=\frac{\lambda_0}{2}$.} Given $x,y\in\rev{\overline{\Sigma_{\lambda_1}}}$ with $d:=|x-y|\rev{<\lambda_1}$, we split
		\[
		|u(x)-u(y)| \le |u(x)-u(x+de_N)| + |u(x+de_N)-u(y+de_N)| + |u(y+de_N)-u(y)|.
		\]
		Along the two vertical segments the height varies in $(0,\rev{\lambda_0})$, so each of the first and third terms is bounded by $C\int_0^{d}\left(1-\ln t\right)^{1/p}\,dt \le C d\left(1+|\ln d|\right)^{1/p}$; along the horizontal segment the height \rev{lies in $[d,\lambda_0)$}, so the middle term is bounded by $Cd\left(1-\ln d\right)^{1/p}$. This gives \eqref{loglip}.
		Estimate \eqref{loglip} is optimal \emph{in order}: by Theorem \ref{th:sharp=1}, $u(x',t)-u(x',0)=u(x',t)=(\kappa_*+o(1))\,t\left(1-\ln t\right)^{1/p}$ as $t\to0^+$. In particular, no coefficient in a two-point estimate can be smaller than $\kappa_*$ asymptotically on vertical boundary increments; we make no claim that the constant $C$ in the full two-point estimate is sharp. Thus $u$ is not Lipschitz continuous up to $\partial\mathbb{R}^N_+$.}
	\end{remark}
	
	\section{Monotonicity of solutions}
	Theorem \ref{th:boundary_bound} is only concerned about the bounds for solutions near $\partial\mathbb{R}^N_+$. The following lemma will provide a global bound from below for solutions.
	\begin{lemma}\label{lem:global_lower_bound}
		Let $g:(0,+\infty)\to\mathbb{R}$ be a locally Lipschitz continuous function such that $g(t) > \frac{c_0}{t^\gamma}$ for all $0<t<t_0$, where $\gamma,c_0,t_0>0$. Let $u \in C^{1,\alpha}_{\rm loc}(\mathbb{R}^N_+) \cap C(\overline{\mathbb{R}^N_+})$ be a solution to the problem
		\[
		\begin{cases}
			-\Delta_p u + \vartheta |\nabla u|^q = g(u) &\text{ in } \mathbb{R}^N_+,\\
			u>0 &\text{ in } \mathbb{R}^N_+,\\
			u=0 &\text{ on } \partial\mathbb{R}^N_+.
		\end{cases}
		\]
		Then there are constants $C>0$ \rev{and $t_2\in(0,t_1]$, depending only on $N$, $p$, $\gamma$, $\vartheta$, $c_0$ and $t_0$,} such that
		\begin{equation}\label{t1_est}
			u(x) \ge \min\left\{C x_N^\frac{p}{\gamma+p-1}, \rev{t_2}\right\} \quad\text{ in } \mathbb{R}^N_+,
		\end{equation}
		where $t_1 := \min\left\{t_0, \left( \frac{c_0}{\vartheta} \right)^{1/\gamma}\right\}$ \new{(with the convention $\left(\frac{c_0}{\vartheta}\right)^{1/\gamma}=+\infty$, hence $t_1=t_0$, when $\vartheta=0$)}.
	\end{lemma}
	
	\begin{proof}
		Let $\lambda_1>0$ be the first eigenvalue and $\phi_1 \in C^1(\overline{B_1})$ a corresponding positive eigenfunction of the $p$-Laplacian in $B_1$, namely,
		\[
		\begin{cases}
			-\Delta_p \phi_1 = \lambda_1 \phi_1^{p-1} & \text{ in } B_1,\\
			\phi_1 > 0 & \text{ in } B_1,\\
			\phi_1 = 0 & \text{ on } \partial B_1.
		\end{cases}
		\]
		
		Setting
		\[
		w = s \phi_1^\frac{p}{\gamma+p-1},
		\]
		where $s>0$ will be chosen later. Direct calculation yields that in the weak sense
		\[
		-\Delta_p w + \vartheta |\nabla w|^p + \vartheta = \frac{ a(x)}{w^\gamma} \quad \text{ in } B_1,
		\]
		where
		\begin{align*}
			a(x)
			&= s^{\gamma + p - 1} \left( \frac{p}{\gamma + p - 1} \right)^{p - 1} \Bigg[ \lambda_1 \phi_1^p\\
			&\qquad + \left( \frac{\vartheta sp}{\gamma + p - 1} \rev{\phi_1^\frac{p}{\gamma+p-1}} + \frac{(\gamma - 1)(p - 1)}{\gamma + p - 1} \right) |\nabla \phi_1|^p \Bigg]\\
			&\quad + \vartheta s^\gamma \phi_1^\frac{\gamma p}{\gamma+p-1}.
		\end{align*}
		
		Now we fix $s>0$ such that $a(x) \le c_0$ in $B_1$. Then let $R_0>0$ be such that
		\[
		R_0^\frac{p}{\gamma+p-1} w(0) = \rev{t_2:=\min\left\{t_1,\,w(0)\right\}},
		\]
		\rev{that is, $R_0=\min\{1,(t_1/w(0))^{\frac{\gamma+p-1}{p}}\}\le1$. The restriction $R_0\le1$ is needed when $\vartheta>0$. Indeed, writing $\theta:=\frac{p}{\gamma+p-1}$, a direct computation gives, for $0<R\le1$,
		\[
		\begin{aligned}
			&-\Delta_p w_{x_0,R} + \vartheta |\nabla w_{x_0,R}|^p + \vartheta\\
			&= R^{-\theta\gamma}\left[-\Delta_p w + R^{\theta}\vartheta|\nabla w|^p + R^{\theta\gamma}\vartheta\right]\left(\frac{x-x_0}{R}\right) \le \frac{c_0}{w_{x_0,R}^\gamma},
		\end{aligned}
		\]
		where $w_{x_0,R}$ is defined below, while for $R>1$ the last inequality may fail.}
		
		For any $0<R\le R_0$ and $x_0 = (x_0', x_{0,N})\in\mathbb{R}^N_+$ with $x_{0,N} \ge R + \varepsilon$, where $\varepsilon$ is sufficiently small, we set
		\[
		w_{x_0,R}(x) := R^\frac{p}{\gamma+p-1}w\left(\frac{x-x_0}{R}\right) \quad\text{ in } B_R(x_0).
		\]
		Then $w_{x_0,R} \le \rev{t_2\le t_1}$ in $B_R(x_0)$. We claim that
		\begin{equation}\label{u<w}
			u \ge w_{x_0,R} \quad\text{ in } B_R(x_0).
		\end{equation}
		Indeed, if \eqref{u<w} does not hold, then the set $\Omega := \{x\in B_R(x_0) \mid u < w_{x_0,R}\}$ is not empty. Moreover,	
		\[
		\begin{cases}
			-\Delta_p w_{x_0,R} + \vartheta |\nabla w_{x_0,R}|^p + \vartheta \le \frac{c_0}{w_{x_0,R}^\gamma} &\text{ in }  \Omega,\\
			-\Delta_p u + \vartheta |\nabla u|^p + \vartheta \ge \frac{c_0}{u^\gamma} &\text{ in }  \Omega,\\
			w_{x_0,R}, u > 0 &\text{ in }  \Omega,\\
			w_{x_0,R} = u &\text{ on } \partial \Omega.
		\end{cases}
		\]
		Lemma \ref{lem:wcp} can be applied to yield $u \ge w_{x_0,R}$ in $\Omega$, which is a contradiction.
		Hence $u \ge w_{x_0,R}$ in $B_R(x_0)$ with $x_{0,N} \ge R + \varepsilon$. Since $\varepsilon>0$ is arbitrary, we deduce
		\[
		u \ge w_{x_0,R} \text{ in } B_R(x_0) \text{ for all } 0<R\le R_0 \text{ and } x_0 \in \mathbb{R}^N_+ \text{ with } x_{0,N} \ge R.
		\]
		
		In particular, if $x_{0,N} = R < R_0$, then
		\[
		u(x_0) \ge w_{x_0,R}(x_0) = w(0)R^\frac{p}{\gamma+p-1} = w(0) x_{0,N}^\frac{p}{\gamma+p-1}.
		\]
		If $x_{0,N} \ge R = R_0$, then
		\[
		u(x_0) \ge w_{x_0,R}(x_0) = w(0)R_0^\frac{p}{\gamma+p-1} = \rev{t_2}.
		\]
		Hence \eqref{t1_est} follows from the fact that $x_0$ is chosen arbitrarily in $\mathbb{R}^N_+$.
	\end{proof}
	
	\new{Before proceeding, we fix the notation used in the moving plane procedure. For $\lambda>0$ we write
	\[
	u_\lambda(x',x_N):=u(x',2\lambda-x_N) \quad\text{ for } x=(x',x_N)\in\Sigma_{2\lambda},
	\]
	for the reflection of $u$ with respect to the hyperplane $\{x_N=\lambda\}$, and, for $0\le a<b$,
	\[
	\Sigma_{a,b}:=\left\{x\in\mathbb{R}^N \mid a<x_N<b\right\},
	\]
	so that $\Sigma_\lambda=\Sigma_{0,\lambda}$.}

	We are ready to present \rev{the proof of Theorem \ref{th:monotonicity}}.
	\begin{proof}[Proof of Theorem \ref{th:monotonicity}]
		Due to Theorem \ref{th:boundary_behavior}, the set
		\[
		\Lambda:=\left\{\lambda>0 \mid u\leq u_\mu \text{ in } \Sigma_\mu \text{ for all } 0<\mu\leq \lambda\right\}
		\]
		is nonempty. Thus, we can define
		\[
		\overline{\lambda}=\sup\Lambda.
		\]
		To obtain the monotonicity of $u$, it suffices to show that $\overline{\lambda}=+\infty$. By contradiction arguments, we assume $\overline{\lambda}<+\infty$. Then $u\leq u_{\overline{\lambda}}$ in $\Sigma_{\overline{\lambda}}$. We can reach a contradiction by showing that for some small $\varepsilon>0$ we have
		\[
		u\leq u_\lambda \text{ in } \Sigma_\lambda \quad\text{ for all } \overline{\lambda} < \lambda<\overline{\lambda} + \varepsilon.
		\]
		
		Due to Theorem \ref{th:boundary_bound} and Lemma \ref{lem:global_lower_bound}, there exists $\tilde{\lambda}>0$ small such that
		\[
		\sup_{\Sigma_{\tilde{\lambda}}} u < \inf_{\mathbb{R}^N_+\setminus\Sigma_{\overline{\lambda}}} u.
		\]
		Therefore, we only need to show that
		\begin{equation}\label{lambda_lambda}
			u\leq u_\lambda \text{ in } \Sigma_\lambda\setminus\Sigma_{\tilde{\lambda}} \quad\text{ for all } \overline{\lambda} < \lambda<\overline{\lambda} + \varepsilon
		\end{equation}
		for some $\varepsilon\in(0,1)$.
		By Lemma \ref{lem:global_lower_bound} again, we know that
		\begin{equation}\label{key1}
			0 < \inf_{\mathbb{R}^N_+\setminus\Sigma_{\tilde{\lambda}}} u \le u, u_\lambda \le \sup_{\Sigma_{2\overline{\lambda}+2}} u \quad\text{ in } \Sigma_\lambda\setminus\Sigma_{\tilde{\lambda}}\new{.}
		\end{equation}
		Hence $g(u)$ and $g(u_\lambda)$ are bounded in $\Sigma_\lambda\setminus\Sigma_{\tilde{\lambda}}$, where
		\[
		g(t) = \frac{1}{t^\gamma} + f(t).
		\]
		Therefore, by standard gradient elliptic estimates (see \cite{MR709038,MR727034}), we have
		\begin{equation}\label{key2}
			|\nabla u|, |\nabla u_\lambda| \in L^\infty(\Sigma_\lambda\setminus\Sigma_{\tilde{\lambda}})
		\end{equation}
		for every $\overline{\lambda} < \lambda<\overline{\lambda} + \varepsilon$. With \eqref{key1} and \eqref{key2} in force, we can repeat the techniques in \cite{MR3303939,MR3752525,MR3118616,MR4439897}, which are based on various comparison principles and compactness arguments for problems with a regular nonlinearity, to prove \eqref{lambda_lambda}. More precisely, if $g$ is positive and $1<p<2$, we use the arguments in \cite{MR3303939}. If $g$ is positive and $p\ge2$, we follow the ones in \cite{MR3752525,MR4750390}.	
		The details for these cases, therefore, will be omitted.
		
		When $g$ is sign-changing and \( \frac{2N+2}{N+2}<p\le2 \), we may apply the techniques in \cite{MR4439897}. However, since \cite{MR4439897} only deals with the $p$-Laplacian operator without a gradient term, we need to modify the corresponding arguments. The details will be given in the appendix for completeness.
		
		\rev{It remains to prove the strict inequality when $p>\frac{2N+2}{N+2}$ and $g>0$ on $(0,+\infty)$. Fix a bounded domain $D\subset\subset\mathbb{R}^N_+$ containing a point of $\Sigma_{\lambda_0}$, where $\lambda_0$ is given by Theorem \ref{th:boundary_behavior}. In $D$, $u$ is bounded from above and from below by positive constants, $g$ is Lipschitz continuous on the range of $u$ and positive there, and $u_{x_N}:=\frac{\partial u}{\partial x_N}\ge0$ is a weak solution of the linearized equation of \eqref{main}. Since $p>\frac{2N+2}{N+2}$ and the nonlinearity is positive, the strong maximum principle for the linearized operator \cite[Theorem 1.3]{MR4022661} applies and yields that either $u_{x_N}>0$ in $D$ or $u_{x_N}\equiv0$ in $D$. The latter is impossible, since $u_{x_N}>0$ in $D\cap\Sigma_{\lambda_0}$ by Theorem \ref{th:boundary_behavior}. Exhausting $\mathbb{R}^N_+$ by an increasing sequence of such domains, we get $u_{x_N}>0$ in $\mathbb{R}^N_+$.}
	\end{proof}
	
	\section{On possibly unbounded solutions in strips}
	Throughout this section, we study solutions $u$ to \eqref{main} which may be unbounded in strips. More precisely, we only assume the following regularity $u \in W^{1,p}_{\rm loc}(\mathbb{R}^N_+)\cap L^\infty_{\rm loc}(\overline{\mathbb{R}^N_+})$.
	
	\subsection{Boundary behavior of solutions}
	We start with a local version of Theorem \ref{th:boundary_bound}.
	
	\begin{theorem}\label{th:boundary_bound'}
		Let $u \in W^{1,p}_{\rm loc}(\mathbb{R}^N_+)\cap L^\infty_{\rm loc}(\overline{\mathbb{R}^N_+})$ be a solution to problem \eqref{main}. Then $u\in C^{1,\alpha}_{\rm loc}(\mathbb{R}^N_+)\cap C(\overline{\mathbb{R}^N_+})$ for some $0<\alpha<1$. Moreover, for any $R>0$, the following assertions hold:
		\begin{enumerate}
			\item[(i)] When $\gamma>1$, there exist $\lambda_0,c_1,c_2>0$ such that
			\[
			c_1 x_N^\frac{p}{\gamma+p-1} \le u(x) \le c_2 x_N^\frac{p}{\gamma+p-1} \quad\text{ in } \Sigma_{\lambda_0} \cap \mathcal{C}_R.
			\]
			\item[(ii)] When $\gamma=1$, there exist $0<\lambda_0<1$, $c_1,c_2>0$ such that
			\[
							c_1 x_N \left(1-\ln x_N\right)^\frac{1}{p} \le u(x) \le c_2 x_N \left(1-\ln x_N\right)^\frac{1}{p} \quad\text{ in } \Sigma_{\lambda_0} \cap \mathcal{C}_R.
				\]
				\rev{Moreover,
				\[
				\lim_{x_N\to0^+} \frac{u(x)}{x_N\left(1-\ln x_N\right)^\frac{1}{p}} = \left(\frac{p}{p-1}\right)^\frac{1}{p}
				\]
				uniformly with respect to $x'\in\mathbb{R}^{N-1}$ with $|x'|<R$.}
				\item[(iii)] When $0<\gamma<1$, there exist $\lambda_0,c_1,c_2>0$ such that
				\[
				c_1 x_N \le u(x) \le c_2 x_N \quad\text{ in } \Sigma_{\lambda_0} \cap \mathcal{C}_R.
			\]
		\end{enumerate}
	\end{theorem}
	
	\begin{proof}
		The proof is similar to that of Theorem \ref{th:boundary_bound}. The main difference is that we argue in \new{compact} sets of $\overline{\mathbb{R}^N_+}$ instead of strips. More precisely, let any $\overline\lambda>0$ and let $x_0$, $\Omega$ be as in \eqref{L}. Since $u \in L^\infty_{\rm loc}(\overline{\mathbb{R}^N_+})$, we can define
		\[
		L=\esssup_{\Sigma_{\overline\lambda} \cap \mathcal{C}_{R+2\overline\lambda}} u.
		\]
		Then we repeat arguments in the proof of Theorem \ref{th:boundary_bound} such that the comparison principle is applied in a subset of $\Sigma_{\overline\lambda} \cap \mathcal{C}_{R+2\overline\lambda}$ instead of $\Sigma_{\overline\lambda}$. For instance, instead of \eqref{M}, we can choose $M$ such that
		\[
		\frac{1}{u^\gamma} + f(u) < \frac{M}{u^\gamma} \quad\text{ for } x\in \Sigma_{\overline\lambda} \cap \mathcal{C}_{R+2\overline\lambda}.
		\]
		Similarly, instead of \eqref{y}, we take
		\[
		y=(y',y_N)\in\Sigma_{\overline\lambda} \cap \mathcal{C}_R,
		\]
		so that $B_{2\overline\lambda}(\hat{y})\cap\Sigma_{\overline\lambda} \subset \Sigma_{\overline\lambda} \cap \mathcal{C}_{R+2\overline\lambda}$.
		Then using the comparison principle, we can derive 
		\[
		u(x) \le C x_N^\frac{p}{\gamma+p-1} \text{ in }  \Sigma_{\overline\lambda} \cap \mathcal{C}_R \quad\text{ for some } C>0
		\]
		instead of \eqref{upper_bound>1}. Next, \eqref{lambda_0} is replaced by
		\[
		\frac{1}{u^\gamma} + f(u) > \frac{1}{2u^\gamma} \quad\text{ for } x\in \Sigma_{2\lambda_0} \cap \mathcal{C}_{R+\lambda_0}
		\]
		and \eqref{y2} is replaced by
		\[
		y=(y',y_N)\in\Sigma_{\lambda_0} \cap \mathcal{C}_R.
		\]
		Then employing the comparison principle on $B_{\lambda_0}(\hat{y})$, we can derive
		\[
		u(x) \ge c x_N^\frac{p}{\gamma+p-1} \text{ in }  \Sigma_{\lambda_0} \cap \mathcal{C}_R \quad\text{ for some } c>0.
		\]
					This proves (i). The other cases are similar\rev{, except for the limit in (ii). For the latter, we apply the two-sided bound in (ii) with $R+1$ in place of $R$ and argue as in the proof of Theorem \ref{th:sharp=1} with $y'$ restricted to $|y'|<R$. As observed at the end of that proof, all comparisons then take place in $\Sigma_{\lambda_0}\cap\mathcal{C}_{R+1}$, where the two-sided bound and \eqref{Meps} hold. Since the constants chosen there do not depend on $y'$, the limit is uniform in $|x'|<R$}.
		\end{proof}
	
	\begin{proof}[Proof of Theorem \ref{th:boundary_behavior'}]
		The proof is similar to that of Theorem \ref{th:boundary_behavior}. We clarify the differences in the case $\gamma>1$.
		
		First, we don't have \eqref{u_bound>1}. But Theorem \ref{th:boundary_bound'} ensures the following local estimate
		\[
		c_1 x_N^\frac{p}{\gamma+p-1} \le u(x) \le c_2 x_N^\frac{p}{\gamma+p-1} \quad\text{ in } \Sigma_{\lambda_0} \cap \mathcal{C}_2.
		\]
		Then, since $x_n\to x_0\in\partial\mathbb{R}^N_+$ \new{(so that $|x_n'-x_0'|\to0$)}, instead of \eqref{wn_bound>1} we only have, for large $n$,
		\[
		c_1 a^\frac{p}{\gamma+p-1} \le c_1 x_N^\frac{p}{\gamma+p-1} \le w_n(x) \le c_2 x_N^\frac{p}{\gamma+p-1} \le c_2 A^\frac{p}{\gamma+p-1} \quad\text{ in } \left(\Sigma_A\setminus \Sigma_a\right) \cap \mathcal{C}_{1/\varepsilon_n}.
		\]
					Nevertheless, this is enough to pass to the limit and \eqref{w_infty}, \eqref{w_infty_bounds} still hold. Then we \rev{conclude} as in the proof of Theorem \ref{th:boundary_behavior}.
			
			This proves the theorem for $\gamma>1$. The other cases can be obtained in the same way\rev{. For the limit in (ii) we use the local sharp constant in Theorem \ref{th:boundary_bound'}(ii) in place of Theorem \ref{th:sharp=1}}.
	\end{proof}
	
	\subsection{Monotonicity of solutions in dimension two}
	Besides Theorem \ref{th:boundary_behavior'}, other necessary tools for proving Theorem \ref{th:monotonicity'} are presented in \cite{MR4876247} (see also \cite{MR2654242,MR4290571,MR4635360,MR4771444,MR4439897}). We recall some notations used in the method of moving planes with geometric techniques in dimension two.
	
	A point in $\mathbb{R}^2$ will be denoted by $(x,y)$.
	For given $(x_0, s, \theta) \in \mathbb{R}\times \mathbb{R}_+ \times \left(-\frac{\pi}{2}, \frac{\pi}{2}\right)$, let $L_{\theta}=(\cos\theta,\sin\theta)$ and let $V_{\theta}$ be the unit vector which is orthogonal to $L_{\theta}$ and satisfies \new{$\langle V_{\theta}, e_2\rangle\ge 0$}.
	\begin{figure}[htp]
		\begin{tikzpicture}[scale=1,>=stealth, font=\footnotesize, line join=round, line cap=round]
			\def\xmin{-5} \def\xmax{5}
			\def\ymin{0} \def\ymax{5}
			\tkzDefPoints{-4/0/L,0/0/A,0/2/S,4/4/B,4/2/C}
			\tkzMarkAngles[size=1cm,arc=l](C,S,B)
			\tkzLabelAngles[pos=1.2](C,S,B){$\theta$}
			\tkzFillPolygon[pattern=north east lines,opacity=0.2](S,L,A)
			\draw[->] (\xmin,0)--(\xmax,0) node [below]{$x$};
			\draw[->] (0,\ymin)--(0,\ymax) node [left]{$y$};
			\fill (0,0) circle (1pt) node[below right]{$x_{0}$};
			\fill (0,2) circle (1pt) node[below right]{$s$};
			\draw[->] (0,2)--(-1,4) node [below left]{$V_{\theta}$};
			\draw (-4,0)--(4,4);
			\draw (-4,2)--(4,2);
			\node at (4,4.6) [below]{$L_{x_{0},s,\theta}$};
			\node at (-1,0.9) [below]{$\mathcal{T}_{x_0,s,\theta}$};
		\end{tikzpicture}
		\caption{The triangle $\mathcal{T}_{x_{0}, s, \theta}$.}
		\label{fig:triangle}
	\end{figure}
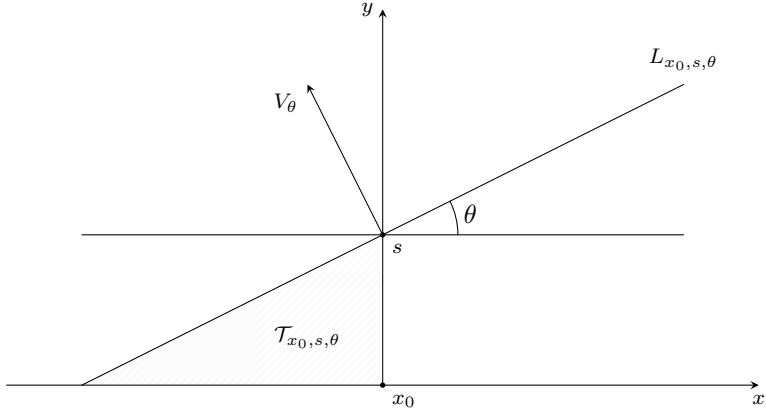
	Besides, we denote by $L_{x_0,s,\theta}$ the line which is parallel to $L_{\theta}$ and passes through $(x_0,s)$. We also denote by $\mathcal{T}_{x_0,s,\theta}$ the triangle delimited by the three lines $L_{x_0,s,\theta}$, $\{x_0\}\times \mathbb{R}$ and $\mathbb{R}\times \{0\}$ (see Figure \ref{fig:triangle}).
	
	Furthermore, for any $x\in\mathcal{T}_{x_0,s,\theta}$, we define
	\[
	u_{x_0,s,\theta}(x)=u(R_{x_0,s,\theta}(x)),
	\]
	where $R_{x_0,s,\theta}(x)$ is the symmetric point of $x$ with respect to $L_{x_0,s,\theta}$. Then $u_{x_0, s, \theta}$ solves
	\[
	-\Delta_p u_{x_0, s, \theta} + \vartheta |\nabla u_{x_0, s, \theta}|^q = \frac{1}{u_{x_0, s, \theta}^\gamma} + f(u_{x_0, s, \theta}) \quad\text{ in } \mathcal{T}_{x_0,s,\theta}
	\]
	in the weak sense.	
	For convenience, we shall denote
	\begin{equation}\label{u_s}
		u_s=u_{x_0, s, 0}.
	\end{equation}
	We also set
	\[
	g(t) := \new{\frac{1}{t^\gamma}} + f(t)
	\]
	and use the following \new{notation}
	\[
	Z_g := \{t\in(0,+\infty) \mid g(t)=0\},
	\]
	\[
	Z_{g(u)} := \{z\in\mathbb{R}^2_+ \mid g(u(z))=0\}.
	\]

\begin{lemma}[Local comparison on a sign interval]\label{lem:local_signed_scp}
 Let $D$ be a connected open set in $\mathbb R^N$, and suppose
 $p>(2N+2)/(N+2)$ and $q\ge\max\{p-1,1\}$. Let $a,b\in C^1_{\rm loc}(D)$
 be positive solutions of
 \[
 -\Delta_p w+\vartheta|\nabla w|^q=g(w),\qquad a\le b\quad\hbox{in }D.
 \]
 Assume that there is an open interval $J\subset\subset(0,\infty)$ on which
 $g$ is locally Lipschitz and has a strict constant sign, and that
 $a(D)\cup b(D)\subset\subset J$. Then either $a<b$ throughout $D$ or $a\equiv b$ in $D$.
 More generally, if $a\le b$ on an open set $U$ and $g(a(x_0))\ne0$ at a
 contact point $x_0\in U$, the contact set is relatively open near $x_0$.
\end{lemma}
\begin{proof}
 Shrink the value interval if necessary and continue the restriction of
 $g$ to a positive locally Lipschitz function on $(0,\infty)$ when $g>0$
 on $J$ (constant continuation at the two ends is sufficient). The
 strong comparison principle of \cite[Theorem 1.2]{MR3859194} then applies
 on $D$. When $g<0$ on $J$, choose $M>\sup_D b$ and set
 \[
 v_1:=M-b\le v_2:=M-a,\qquad
 -\Delta_p v_i-\vartheta|\nabla v_i|^q=-g(M-v_i)>0.
 \]
 Both $v_i$ are positive; the right-hand side has a positive locally
 Lipschitz continuation from their compact value interval to $(0,\infty)$.
 Apply the same theorem, now with the constant first-order coefficient
 $-\vartheta$: its structural hypothesis requires this coefficient to be
 locally Lipschitz, not nonnegative. Indeed, the assumptions $(hp^*)$ of
 \cite[Theorem 1.2]{MR3859194} place the positivity restriction on the
 reaction, not on the first-order coefficient. This proves the alternative. At any contact point where
 $g(a(x_0))\ne0$, continuity makes both solution values lie in a smaller
 compact interval of constant sign on a sufficiently small ball, proving
 the final, local assertion. Notice that the singularity of $g$ at zero
 plays no role in this interior argument.
\end{proof}

	The weak comparison principle below is stated for dimension $N\ge2$.
	In what follows, we write a point $x\in\mathbb{R}^N$ as $x=(x',x_N)\in\mathbb{R}^{N-1}\times\mathbb{R}$. For any set $\Omega\subset\mathbb{R}^N$ we denote $\Omega'$ the projection of $\Omega$ on $\mathbb{R}^{N-1}$ in the $x_N$-direction, i.e.,
	\[
	\Omega':=\{x'\in\mathbb{R}^{N-1} \mid (x',y)\in\Omega\text{ for some }y\in\mathbb{R}\}.
	\]
	The open ball of center $x_0$ with radius $r>0$ is always denoted as $B_r(x_0)$.
	
	We recall some tools to be used in the case $p<2$. First of all, we have the following weak comparison principle.
	
\begin{lemma}[Interior Lipschitz truncation]\label{lem:clip_g}
 Let $E\subset\subset\mathbb R^2_+$ be a bounded set containing the closure of a
 comparison domain and the reflections of its points for all positions of
 the reflection line under consideration. Suppose
 $0<m\le u\le M$ on $E$. Set
 \[ 
 P_{m,M}(t):=\max\{m,\min\{t,M\}\},\qquad
 g_{m,M}(t):=g(P_{m,M}(t)),\quad t\in\mathbb R.
 \]
 Then $g_{m,M}$ is globally Lipschitz on $\mathbb R$ and equals $g$
 on $[m,M]$. Thus, in that comparison domain, both $u$ and its reflected
 counterpart solve the same equation with $g_{m,M}$ in place of $g$.
 If $g>0$ on $(0,\infty)$, then $g_{m,M}>0$ on $\mathbb R$ as well.
 The Lipschitz constant and the bounds can be selected uniformly for a
 compact range of positions of the reflection line, provided one fixed
 set $E\subset\subset\mathbb R^2_+$ contains all corresponding comparison
 domains and their reflections.
\end{lemma}
\begin{proof}
 The compact interval $[m,M]$ stays away from the singular point $0$,
 so $g$ is Lipschitz there. The projection $P_{m,M}$ is $1$-Lipschitz;
 the other assertions are immediate. The requirement that $E$ also
 contains the reflected points is essential: a bound on $u$ alone in
 the unreflected comparison domain would not suffice.
\end{proof}

\begin{theorem}[Theorem 10 in \cite{MR4876247}]\label{th:wcp_p<2}
		Assume $1<p<2$, $q\ge1$ and $g$ is a locally Lipschitz continuous function.
		Let $\Omega \subset \mathbb{R}^{N}$ be a bounded domain, $L,M>0$ and $u, v \in C^1(\overline{\Omega})$ be such that
		\[
		\Omega\subset\{x\in\mathbb{R}^N \mid |x_N|\le L\},
		\]
		\[
		\|u\|_{L^{\infty}(\Omega)} + \|v\|_{L^{\infty}(\Omega)} + \|\nabla u\|_{L^{\infty}(\Omega)} + \|\nabla v\|_{L^{\infty}(\Omega)} \le M
		\]
		and
		\[
		\begin{cases}
			-\Delta_{p} u+\vartheta|\nabla u|^{q} \le g(u) &\text { in } \Omega, \\
			-\Delta_{p} v+\vartheta|\nabla v|^{q} \ge g(v) &\text { in } \Omega, \\
			u>0, v>0 &\text { in } \Omega, \\
			u \le v &\text { on } \partial \Omega,
		\end{cases}
		\]
		Assume further that
		\[
		\Omega = \bigcup_{x'\in\Omega'} \{x'\}\times (A_{x'} \cup B_{x'}),
		\]
		where measurable sets $A_{x'},B_{x'} \subset (-L,L)$ satisfy
		\[
		|A_{x'}|\le\delta \quad \text{ and } \quad B_{x'} \subset \{x_N \in (-L,L) \mid |\nabla u(x',x_N)| + |\nabla v(x',x_N)|\le \delta\}.
		\]
		Then there exists a constant
		\[
		\delta_0 = \delta_0 (N, p, q, \vartheta, g, M, L)
		\]
		such that if we assume $\delta \le \delta_0$, then it holds
		\[
		u \le v \text { in } \Omega.
		\]
	\end{theorem}
	
	The following lemma compensates for the lack of a strong comparison principle for quasilinear elliptic equations with sign-changing nonlinearity. Together with Theorem \ref{th:wcp_p<2}, it will play a vital role in our later argument.	
	\begin{lemma}[Lemma 11 in \cite{MR4876247}]\label{lem:scp1}
		Under the assumptions of Theorem \ref{th:monotonicity'} (i), let us assume that
		\[
		\begin{cases}
			\frac{\partial u}{\partial V_{\theta}} \ge 0 &\text{ in } \mathcal{T}_{x_0,s,\theta},\\
			u\le u_{x_0, s, \theta} &\text{ in } \mathcal{T}_{x_0,s,\theta},\\
			u< u_{x_0, s, \theta} &\text{ on } \partial\mathcal{T}_{x_0,s,\theta} \setminus L_{x_0,s,\theta}
		\end{cases}
		\]
		for some $(x_0, s, \theta) \in \mathbb{R}\times \mathbb{R}_+ \times \left(-\frac{\pi}{2}, \frac{\pi}{2}\right)$.
		Then
		\begin{equation}\label{pos_deri}
			u < u_{x_0, s, \theta} \quad\text{ in } \mathcal{T}_{x_0,s,\theta} \setminus (Z_u\cap Z_{u_{x_0, s, \theta}}),
		\end{equation}
		where
		\[
		Z_u=\{z\in\mathbb{R}^2_+ \mid |\nabla u(z)|=0\}
		\]
		and
		\[
		Z_{u_{x_0, s, \theta}}=\{z\in R_{x_0,s,\theta}(\mathbb{R}^2_+) \mid |\nabla u_{x_0, s, \theta}(z)|=0\}.
		\]
	\end{lemma}
	
\begin{proof}[Justification of Lemma \ref{lem:scp1} for singular $g$]
 Write $v=u_{x_0,s,\theta}$ and $T=\mathcal T_{x_0,s,\theta}$.
 The argument of \cite[Lemma 11]{MR4876247} only uses local
 Lipschitz continuity of the reaction at the \emph{positive values} met
 during the interior continuation. We detail why its singularity at zero
 does not affect that argument. Since $f$ is bounded near zero,
 $g(t)=t^{-\gamma}+f(t)>0$ for $0<t<t_*$ with some $t_*>0$.
 Moreover, $u$ and $v$ are bounded on the closure of the triangle, so
 the discrete set $Z_g$ contains only finitely many zeroes in their
 combined attained value range. List these zeroes as
 $0<z_1<\cdots<z_m$, put $z_0=0$ and $z_{m+1}=+\infty$;
 the argument below also covers $m=0$.

 Suppose by contradiction that $u(z)=v(z)$ at a point $z\in T$ where
 $|\nabla u(z)|+|\nabla v(z)|>0$. On a sufficiently small ball about
 $z$, the values of both solutions belong to a compact subinterval of
 $(0,+\infty)$; hence $g$ is Lipschitz there. The nondegenerate local
 comparison principle used in \cite[Lemma 11]{MR4876247} gives
 $u\equiv v$ on a smaller ball. Their gradients agree and are nonzero
 there, so their common (nonconstant) values cannot all belong to the
 discrete set $Z_g$. We can therefore choose a contact point $z'$ such
 that $z_k<u(z')<z_{k+1}$ for some $k\ge0$.

 Let $\Omega_k$ be the component of
 $T\cap\{z_k<u<z_{k+1}\}$ containing $z'$. The local assertion of
 Lemma \ref{lem:local_signed_scp} shows that the contact set
 $\{u=v\}\cap\Omega_k$ is relatively open as well as closed;
 consequently $u=v$ on $\Omega_k$. If $m=0$, this already implies
 equality throughout $T$, contrary to the strict comparison on
 $\partial T\setminus L_{x_0,s,\theta}$.

 Otherwise translate a small closed ball in $\Omega_k$ in direction
 $-V_\theta$ until its first contact with the boundary of $\Omega_k$.
 The monotonicity $\partial_{V_\theta}u\ge0$ rules out an exit through
 the upper level $u=z_{k+1}$. Geometrically the first contact is either
 on $\partial T\setminus L_{x_0,s,\theta}$, where the assumed strict
 inequality contradicts the propagated equality, or at a positive
 lower level $u=z_k$. If $k=0$, the latter level is zero and lies on the
 bottom boundary, again giving the same contradiction. No differential
 estimate at $u=0$ is invoked here.

 If $k\ge1$, let $B$ be the translated ball, tangent at a point $\hat z$
 with $u(\hat z)=z_k>0$, and set $w=u-z_k>0$ in $B$. Because
 $g(z_k)=0$, $g$ is Lipschitz near $z_k$, and $1<p\le2$, a sufficiently
 large constant $C$ gives
 \[
 -\Delta_p w+\vartheta|\nabla w|^q+Cw^{p-1}
     =g(z_k+w)+Cw^{p-1}\ge0 \qquad\text{in }B.
 \]
 The Hopf lemma at $\hat z$ yields $\nabla u(\hat z)\ne0$ and a
 direction entering the lower value interval
 $(z_{k-1},z_k)$. The local nondegenerate comparison principle
 propagates $u=v$ across $\hat z$, producing a contact point in that
 lower interval. Repeat the same finite descent in $k$. Eventually a
 contact in $(0,z_1)$ produces the bottom-boundary contradiction.
 All estimates in this descent are taken at positive levels; the
 singularity at zero is used only to guarantee $g>0$ on $(0,t_*)$.
 Thus \eqref{pos_deri} follows.
\end{proof}
	
	We recall that the strong maximum principle for the linearized operator does not hold for sign-changing nonlinearities. However, the following weaker result will suffice for our purpose.
	\begin{lemma}[Lemma 12 in \cite{MR4876247}]\label{lem:scp2}
		Under the assumptions of Theorem \ref{th:monotonicity'} (i), let us assume that for some $\lambda>0$ we have
		\[
		u \le u_{\new{\tau}} \quad\text{ in } \Sigma_{\new{\tau}} \text{ for all } \new{\tau} \in (0,\lambda],
		\]
		where $u_{\new{\tau}}$ is defined in \eqref{u_s}.
		Then for every interval $I\subset\mathbb{R}$, there exists $\overline{x}\in I$ such that
		\[			
		\begin{aligned}
			\frac{\partial u}{\partial y}>0 &\text{ on } \{\overline{x}\} \times (0,\lambda),\\
			u<u_{\new{\tau}} &\text{ on } \{\overline{x}\} \times [0,\new{\tau}) \text{ for all } \new{\tau} \in (0,\lambda].
		\end{aligned}
		\]
		Moreover,
		\[
		\frac{\partial u}{\partial y}>0 \quad\text{ in } \Sigma_{\lambda} \setminus Z_{g(u)}.
		\]
	\end{lemma}
\begin{proof}[Proof of Lemma \ref{lem:scp2} for singular $g$]
 The reflection inequalities, applied at $\tau=y+h$ and then letting
 $h\downarrow0$, give $u_y\ge0$ in $\Sigma_\lambda$.
 Write $Z_g=\{z_1<z_2<\cdots\}$ if the zero set is nonempty,
 with $z_0=0$; on any bounded interval this list is finite. For a finite
 list we set the next zero to $+\infty$. If the list is empty, the
 arguments below take place in the sole sign interval $(0,\infty)$.
 Notice that $g(t)>0$ for $0<t<t_*$, because $f$ is bounded near zero.

 First let $\Omega$ be a component of
 $\Sigma_\lambda\setminus Z_{g(u)}$. On compact subsets of
 $\Omega$, the values of $u$ are bounded away from zero and $g$ has
 a strict constant sign. For $g>0$, the strong maximum principle for
 the linearized equation \cite[Theorem 1.3]{MR4022661} applies locally
 after a positive Lipschitz continuation of $g$ on the attained range.
 For $g<0$, put $v=M-u>0$ locally. Then
 \[
 -\Delta_pv-\vartheta|\nabla v|^q=-g(M-v)>0,
 \]
 and $u_y=-v_y\ge0$ is a nonnegative solution of the homogeneous
 linearized equation for $v$. The same principle applies, since the
 first-order coefficient is allowed to have either sign. Exhaustion
 of $\Omega$ gives the alternative $u_y>0$ in $\Omega$ or
 $u_y\equiv0$ there. The latter is impossible: a vertical segment
 followed downward from a point of $\Omega$ would keep a fixed
 positive value of $u$ until its first exit, whereas that exit must
 occur either at $y=0$ (where $u=0$) or at $g(u)=0$ (where its constant
 value would already put the starting point outside $\Omega$).
 Consequently,
 \begin{equation}\label{positive_derivative_sign_regions}
 u_y>0\qquad\text{in }\Sigma_\lambda\setminus Z_{g(u)}.
 \end{equation}

 We explain the continuation across the finitely many positive zero
 levels, following the sliding-ball construction in
 \cite[proof of Lemma 12]{MR4876247}. Given an interval $I$, choose
 $x_0\in I$ and $r_0>0$ so small that
 $\overline{B_{r_0}((x_0,r_0))}\subset
 (I\times[0,2r_0])\cap\overline{\Sigma_\lambda}$ and
 $0<u<z_1$ in its interior (the upper constraint is void if there are
 no zeroes). Such a choice is possible by the local boundary continuity
 and the Dirichlet condition. The initial ball lies in the first sign
 component, hence $u_y>0$ there by
 \eqref{positive_derivative_sign_regions}.

 Translate the ball vertically upwards while it stays within the same
 component of $\{z_j<u<z_{j+1}\}\cap\Sigma_\lambda$.
 It cannot first leave through the lower value level: the translation
 is upwards and $u_y\ge0$. If it reaches the line $y=\lambda$,
 the translating ball together with its downward swept tube provides
 a vertical column of strict positivity from the bottom to that line.
 Otherwise the first contact is at an interior point $\hat z$ with
 $u(\hat z)=z_{j+1}>0$. In the terminal ball $B$ we have
 $w:=z_{j+1}-u>0$, $w(\hat z)=0$. The \emph{correct sign} after
 reversing the dependent variable is
 \[
 -\Delta_p w-\vartheta|\nabla w|^q
       =-g(z_{j+1}-w).
 \]
 As $g(z_{j+1})=0$ and $g$ is Lipschitz near the positive level
 $z_{j+1}$, for sufficiently large $C$, using $1<p\le2$ and
 boundedness of $w$ in $B$, we obtain
 \[
 -\Delta_p w+Cw^{p-1}
    \ge\vartheta|\nabla w|^q-g(z_{j+1}-w)+Cw^{p-1}
    \ge0\qquad\text{in }B.
 \]
 (Equivalently, the equality in the preceding display may be read as
 $-\Delta_p w+Cw^{p-1}
 =\vartheta|\nabla w|^q-g(z_{j+1}-w)+Cw^{p-1}$.)
 The standard Hopf boundary lemma therefore yields
 $|\nabla u(\hat z)|>0$. We do not infer $u_y(\hat z)>0$
 from the direction of the ball normal alone: instead, in a small
 neighbourhood of $\hat z$ the linearized equation is uniformly
 elliptic because $|\nabla u|>0$, and its ordinary strong maximum
 principle, applied to $u_y\ge0$ (which is positive on the portion
 inside $B$), gives $u_y(\hat z)>0$.

 By continuity, a smaller ball immediately above $\hat z$ has
 $z_{j+1}<u<z_{j+2}$ and $u_y>0$. Choose it small enough that its
 downward swept tube lies within the preceding swept tube together
 with the neighbourhood of $\hat z$ just obtained; this is the
 overlapping-ball construction detailed in
 \cite[proof of Lemma 12, Figure 3]{MR4876247}. Iterate. The balls
 remain in the bounded rectangle over a compact subinterval of $I$
 and below $y=\lambda$, where $u$ is bounded. Discreteness of $Z_g$
 therefore implies that only finitely many zero levels can be crossed.
 At the last stage a ball reaches $y=\lambda$. The nested downward
 tubes yield some $\bar x\in I$ for which
 \[
 u_y(\bar x,y)>0\qquad (0<y<\lambda).
 \]
 Here the construction never evaluates the derivative at $y=0$.

 Finally fix $0<\tau\le\lambda$. If the given inequality
 $u\le u_\tau$ had a contact at $(\bar x,y_0)$ with
 $0<y_0<\tau$, then at this minimum of $u_\tau-u$ the gradients
 agree. Since $u_y(\bar x,y_0)>0$, both gradients are nonzero.
 The nondegenerate local strong comparison principle yields equality
 in a neighbourhood of the contact point. The set of contact heights
 on the line $x=\bar x$ is both open and closed in $(0,y_0]$:
 at every such contact $u_y(\bar x,y)>0$, so the same local principle
 applies. Therefore equality propagates along the whole vertical
 segment down to $y=0$. Taking the boundary limit gives
 $u(\bar x,0)=u_\tau(\bar x,0)$, contrary to
 $0=u(\bar x,0)<u(\bar x,2\tau)$. Thus strictness holds for
 $0<y<\tau$, and at $y=0$ it follows directly from positivity.
 No derivative at the singular boundary is used.
\end{proof}
	
	Lastly, we present a weak comparison principle for small domains in the case $p>2$.
	\begin{theorem}[Theorem 13 in \cite{MR4876247}]\label{th:wcp_p>2}		
		Assume $p>2$, $q \ge \frac{p}{2}$ and $g$ is a locally Lipschitz continuous function such that
		\[
		g(t)>0 \text{ for } t>0.
		\]
		Let $\Omega \subset \mathbb{R}^{N}$ be a bounded domain, $L,M>0$ and $u, v \in C^1(\overline{\Omega})$ be such that
		\[
		\|u\|_{L^{\infty}(\Omega)} + \|v\|_{L^{\infty}(\Omega)} + \|\nabla u\|_{L^{\infty}(\Omega)} + \|\nabla v\|_{L^{\infty}(\Omega)} \le M
		\]
		and
		\[			
		\begin{cases}
			-\Delta_{p} u+\vartheta|\nabla u|^{q} \le g(u) &\text { in } \Omega, \\
			-\Delta_{p} v+\vartheta|\nabla v|^{q} \ge g(v) &\text { in } \Omega, \\
			u>0, v>0 &\text { in } \Omega, \\
			u \le v &\text { on } \partial \Omega.
		\end{cases}
		\]
		Then there exists a constant
		\[
		\delta_0 = \delta_0 (N, p, q, \vartheta, g, M)
		\]
		such that if we assume $|\Omega| \le \delta_0$, then it holds
		\[
		u \le v \text { in } \Omega.
		\]
	\end{theorem}
	
	\begin{remark}\label{rem:comparison_clipping}
		Every invocation of Theorem \ref{th:wcp_p<2} or
        \ref{th:wcp_p>2} below takes place in a domain with closure
        in $\mathbb R^2_+$, after removal of a boundary collar. Before
        invoking it, we replace the singular reaction $g$ by the fixed
        extension $g_{m,M}$ of Lemma \ref{lem:clip_g}, with $[m,M]$
        containing the ranges of both $u$ and its reflection throughout
        the compact parameter interval used in sliding. In particular,
        the comparison threshold is uniform in the moving parameter.
        For Theorem \ref{th:wcp_p>2}, the extension is still strictly
        positive. For $p=2$ the small-domain principle follows instead
        by the classical linear energy estimate, as detailed at the end
        of Case (i).
	\end{remark}
	
	We have collected enough tools to present a proof of Theorem \ref{th:monotonicity'} as follows.
	\begin{proof}[Proof of Theorem \ref{th:monotonicity'}] 
		By Theorem \ref{th:boundary_bound'}, we have that $u\in C^{1,\alpha}_{\rm loc}(\new{\mathbb{R}^2_+})\cap C(\new{\overline{\mathbb{R}^2_+}})$ for some $0<\alpha<1$. We consider two cases.
		
		\textbf{Case (i):} $\frac{3}{2} < p \le 2$, $1< q\le p$ and $Z_g$ is a discrete set.
		
		Fix some $x_0\in \mathbb{R}$.
		By Theorem \ref{th:boundary_behavior'}, there exists $h>0$ such that
		\begin{equation}\label{start_deri}
			\frac{\partial u}{\partial V_{\theta}}>0 \quad\text{ in } (x_0-h,x_0+h)\times(0,2h)
		\end{equation}
		for any $|\theta|\le \frac{\pi}{4}$.		
		To show that $\frac{\partial u}{\partial y}\ge0 \text{ in } \mathbb{R}^2_+$, we carry out the moving planes procedure in three steps.
		
		\textit{Step 1.} We show that
		\begin{equation}\label{s1}
			u \le u_{\lambda} \quad\text{ in } \Sigma_{\lambda}
		\end{equation}
		for every $0<\lambda\le h$.
		
		Clearly, for each $0<|\theta|\le \frac{\pi}{4}$, there exists $s_\theta \in (0,h]$ such that
		\begin{align*}
			\mathcal{T}_{x_0,s,\theta} \cup R_{x_0,s,\theta}(\mathcal{T}_{x_0,s,\theta}) \subset (x_0-h,x_0+h)\times(0,2h) \quad \text{ for all } 0<s\le s_\theta.
		\end{align*}
		Moreover, \eqref{start_deri} implies
		\[
		u<u_{x_0,s,\theta} \quad\text{ in } \mathcal{T}_{x_0,s,\theta}
		\]
		for all $0<|\theta|\le \frac{\pi}{4}$ and $0< s\le s_\theta$.
		Denoting
		\begin{align*}
			S_\theta=\{\tilde{s}\in(0,h] \mid u\le u_{x_0,s,\theta} \text{ in } \mathcal{T}_{x_0,s,\theta} \text{ for every } 0< s\le \tilde{s}\}.  
		\end{align*}
		Since $S_\theta\neq \emptyset$, we may set $\overline{s}:=\sup S_\theta\le h$. Using the sliding technique, we claim that
		\begin{equation}\label{s1'}
			\overline{s}=h.
		\end{equation}
		
		Assume on contrary that $\overline{s}<h$. By the definition of $\overline{s}$, we have
		\[
		\frac{\partial u}{\partial V_{\theta}} \ge 0 \quad\text{ and }\quad u\le u_{x_0,\overline{s},\theta} \quad \text{ in } \mathcal{T}_{x_0,\overline{s},\theta}.
		\]
		On the other hand, from \eqref{start_deri} and the Dirichlet boundary condition, we have
		\begin{equation}\label{boundary_cond}
			u< u_{x_0, s, \theta} \quad\text{ on } \partial\mathcal{T}_{x_0,s,\theta} \setminus L_{x_0,s,\theta} \quad\text{ for all } 0< s\le h.
		\end{equation}
		
		Therefore, we can apply Lemma \ref{lem:scp1} to deduce
		\begin{equation}\label{s1''}
			u < u_{x_0,\overline{s},\theta} \quad\text{ in } \mathcal{T}_{x_0,\overline{s},\theta} \setminus (Z_u\cap Z_{u_{x_0, \overline{s}, \theta}}).
		\end{equation}
		
		Now we take $R_0$ sufficient large such that
		\[
		\mathcal{T}_{x_0,s,\theta} \cup R_{x_0,s,\theta}(\mathcal{T}_{x_0,s,\theta}) \subset (x_0-R_0, x_0+R_0) \times (0,R_0) \quad\text{ for all } 0< s\le h.
		\]		
		By Theorem \ref{th:boundary_behavior'}, there exists $\rho>0$ such that
		\begin{equation}\label{start_deri'}
			\frac{\partial u}{\partial V_{\theta}}>0 \quad\text{ in } (x_0-R_0, x_0+R_0) \times(0,\rho)
		\end{equation}
		for any $|\theta|\le \frac{\pi}{4}$. Let $0<\rho'<\rho$ such that
		\begin{equation}\label{infsup}
			\inf_{(x_0-R_0, x_0+R_0) \times(\rho, R_0)} u > \sup_{(x_0-R_0, x_0+R_0) \times(0,\rho')} u.
		\end{equation}
        Fix in advance the closed rectangle
        $E:=[x_0-R_0,x_0+R_0]\times[\rho'/2,R_0]$, which has positive
        distance from the bottom boundary. All triangles for $s\le h$
        and their reflected images were enclosed in the chosen box;
        moreover the reflected height is at least the original height.
        Thus $E$ contains the closures of all comparison domains after
        removing $\Sigma_{\rho'}$, as well as their reflections.
        Put $m_*:=\min_Eu>0$, $M_*:=\max_Eu<\infty$, and
        $G_*:=g_{m_*,M_*}$ as in Lemma \ref{lem:clip_g}.
        The fixed bounds needed in weak comparison are
		\[
		M=2\|u\|_{L^{\infty}(E)} + 2\|\nabla u\|_{L^{\infty}(E)},
		\]
        finite by interior regularity \cite{MR709038,MR727034}.
        If $1<p<2$, select $\delta_0$ from Theorem
        \ref{th:wcp_p<2} for the \emph{globally Lipschitz} source
        $G_*$, bounds $M$, and height $L=h$. If $p=2$, select instead
        the uniform small-area threshold furnished by the linear
        energy estimate at the end of Case (i), with the same $G_*$ and
        $M$. In particular, $\delta_0$ is chosen \emph{after} fixing
        the truncation, and does not depend on $s$.
		
		We choose a sufficiently large compact set $K \subset \rev{\left(\mathcal{T}_{x_0,\overline{s},\theta}\cup(\{x_0\}\times(0,\overline s))\right)} \setminus (Z_u\cap Z_{u_{x_0, \overline{s}, \theta}})$ \rev{(it is convenient to let $K$ contain a neighbourhood, relative to $\overline{\mathcal{T}_{x_0,\overline{s},\theta}}$, of a large compact part of the vertical side, since the vertical fibres close to that side are not short)} so that if we denote $\Omega^s:=\mathcal{T}_{x_0,s,\theta} \setminus K$, then $\Omega^{\overline{s}}$ can be decomposed as
		\[
		\Omega^{\overline{s}} = \bigcup_{x'\in(\Omega^{\overline{s}})'} \{x'\}\times (A_{x'}^{\overline{s}} \cup B_{x'}^{\overline{s}}),
		\]
		where $A_{x'}^{\overline{s}}, B_{x'}^{\overline{s}}$ satisfy
		\[
		|A_{x'}^{\overline{s}}|\le\frac{\delta_0}{2} \quad \text{ and } \quad B_{x'}^{\overline{s}} \subset \left\{x_N > 0 \mid |\nabla u(x',x_N)| + |\nabla \new{u_{x_0,\overline{s},\theta}}(x',x_N)|\le \frac{\delta_0}{2}\right\}.
		\]
		
				From \eqref{s1''} \rev{and \eqref{boundary_cond}}, we know that
		\[
		u \le u_{x_0,\overline{s},\theta} - C \text{ in } K \quad\text{ for some } C>0.
		\]
		By continuity, there exists $0<\varepsilon_0<h-\overline{s}$ such that for any $\overline{s}<s<\overline{s}+\varepsilon_0$, we have
		\begin{equation}\label{K_ineq}
			u < u_{x_0,s,\theta} \quad \text{ in } K
		\end{equation}
		and
		\[
		\Omega^s = \bigcup_{x'\in(\Omega^s)'} \{x'\}\times (A_{x'}^s \cup B_{x'}^s)
		\]
		with
		\[
		|A_{x'}^s|\le\delta_0 \quad \text{ and } \quad B_{x'}^s \subset \{x_N > 0 \mid |\nabla u(x',x_N)| + |\nabla \new{u_{x_0,s,\theta}}(x',x_N)|\le\delta_0\}.
		\]
		
		In view of \eqref{infsup} and \eqref{start_deri'}, we have
		\begin{equation}\label{K_ineq1}
			u \le u_{x_0,s,\theta} \quad \text{ in } \mathcal{T}_{x_0,s,\theta} \cap \Sigma_{\rho'}.
		\end{equation}
		
		From \eqref{boundary_cond}, \eqref{K_ineq} and \eqref{K_ineq1}, we
		have $u \le u_{x_0,s,\theta}$ on $\partial(\mathcal{T}_{x_0,s,\theta} \setminus (K\cup \Sigma_{\rho'}))$.
        Both functions have the same regular reaction $G_*$ on the
        cut-off domain by the already fixed rectangle $E$. Hence, for
        $1<p<2$, Theorem \ref{th:wcp_p<2}, with the threshold
        $\delta_0$ selected above, applies to
        $v=u_{x_0,s,\theta}$ and
        $\Omega=\Omega^s\setminus\overline{\Sigma_{\rho'}}$
        and gives the desired comparison. When $p=2$, the same comparison follows from the
        classical small-domain argument proved below, without using
        Theorem \ref{th:wcp_p<2}. In either case,
		\begin{equation}\label{K_ineq2}
			u \le u_{x_0,s,\theta} \quad \text{ in } \mathcal{T}_{x_0,s,\theta} \setminus (K\cup \Sigma_{\rho'}).
		\end{equation}
		
		Combining \eqref{K_ineq}, \eqref{K_ineq1} and \eqref{K_ineq2}, we get
		\[
		u \le u_{x_0,s,\theta} \quad \text{ in } \mathcal{T}_{x_0,s,\theta}
		\]
		for all $\overline{s}<s<\overline{s}+\varepsilon_0$. This contradicts the definition of $\overline{s}$ and \eqref{s1'} is proved. Hence for every $0<\lambda\le h$ and $0<|\theta|\le \frac{\pi}{4}$, we have
		\[
		u\le u_{x_0,\lambda,\theta} \quad\text{ in } \mathcal{T}_{x_0,\lambda,\theta}.
		\]
		
		Taking the limit as $\theta\to0^+$ and $\theta\to0^-$, we obtain \eqref{s1}. Hence the set
		\[
		\Lambda=\{\tilde{\lambda}>0 \mid u \le u_{\lambda}\text{ in } \Sigma_{\lambda} \text{ for all } 0<\lambda\le\tilde{\lambda}\}
		\]
		is nonempty.
		
		\textit{Step 2.} Setting $\overline{\lambda}:=\sup\Lambda$, we claim that
		\begin{equation}\label{s2}
			\overline{\lambda}=\infty.
		\end{equation}
		
		Assume by contradiction that $\overline{\lambda}<\infty$. From the definition of $\overline{\lambda}$ we have
		\[
		u \le u_\lambda \text{ in } \Sigma_\lambda \quad \text{ for all } \lambda \in (0,\overline{\lambda}]
		\]
		and hence
		\[
		\frac{\partial u}{\partial y} \ge 0 \quad\text{ in } \Sigma_{\overline{\lambda}}.
		\]
		
		We show that there exists some $\overline{x}\in\mathbb{R}$ such that
		\begin{equation}\label{post_bound}
			\frac{\partial u}{\partial y}(\overline{x},\overline{\lambda})>0.
		\end{equation}
		Suppose by contradiction that $\frac{\partial u}{\partial y}(x,\overline{\lambda})= 0$ for all $x\in\mathbb{R}$. There are two possibilities:
		
		\textit{- Possibility 1:} $g(u(x,\overline{\lambda})) = 0$ for all $x\in\mathbb{R}$. Since $Z_g$ is discrete, we deduce $u(x,\overline{\lambda}) = \mu$ for all $x\in\mathbb{R}$ and some $\mu\in Z_g$. Then $w := \mu - u \ge 0$ in $\Sigma_{\overline\lambda}$. Moreover,
		\[
		-\Delta_p w - \vartheta|\nabla w|^q + C w^{p-1} = -g(\mu-w) + C w^{p-1} \ge 0 \quad\text{ in } B_{\frac{\overline{\lambda}}{2}}((x,\overline{\lambda}/2))
		\]
		for sufficiently large $C$. In particular,
        $-\Delta_p w+Cw^{p-1}\ge0$ because $\vartheta\ge0$;
        the sign of the first-order term has therefore been accounted
        for. Hence the strong maximum principle (see
        \cite[Theorems 2.5.1 and 5.5.1]{MR2356201}) implies $w\equiv0$ or $w>0$ in $B_{\frac{\overline{\lambda}}{2}}((x,\overline{\lambda}/2))$. The former cannot happen because $w(x,0)=\mu-u(x,0)=\mu>0$. Hence $w>0$ in $B_{\frac{\overline{\lambda}}{2}}((x,\overline{\lambda}/2))$. Now the Hopf lemma yields $\frac{\partial w}{\partial y}(x,\overline{\lambda}) < 0$, which means $\frac{\partial u}{\partial y}(x,\overline{\lambda}) > 0$ for all $x\in\mathbb{R}$\new{, a contradiction with our assumption}.
		
		\textit{- Possibility 2:} $g(u(x_0,\overline{\lambda})) \ne 0$ for some $x_0 \in \mathbb{R}$. By continuity, there exists small $r>0$ such that $g(u)$ is either strictly positive or strictly negative in $B_r((x_0,\overline{\lambda}))$. Let us define
		\[
		u_{*}(x,y)=
		\begin{cases}
			u(x,y)&\text{ if } 0\le y\le \overline{\lambda},\\
			u\left(x,2\overline{\lambda}-y\right)&\text{ if } \overline{\lambda}\le y\le 2\overline{\lambda},
		\end{cases}
		\]
		and
		\[
		u^{*}(x,y)=
		\begin{cases}
			u(x,2\overline{\lambda}-y)&\text{ if } 0\le y\le \overline{\lambda},\\
			u(x,y)&\text{ if } \overline{\lambda}\le y\le 2\overline{\lambda}.
		\end{cases}
		\]
		Since $u\in C^1(\mathbb{R}^2_+)$ and $\frac{\partial u}{\partial y}=0$ on $\mathbb{R}\times\{\overline{\lambda}\}$, we get that $u_{*}, u^{*}\in C^1_{\rm loc}(\Sigma_{2\overline\lambda})$ and $u_{*}, u^{*}$ are weak solutions of
		\[
		-\Delta_p w + \vartheta|\nabla w|^q = g(w) \quad\text{ in } B_r((x_0,\overline{\lambda})).
		\]
		On the other hand, $u_{*}\le u^{*}$ in $\Sigma_{2\overline{\lambda}}$ and $u_{*}=u^{*}$ on $\mathbb{R}\times\{\overline{\lambda}\}$. We reduce $r$ if necessary so that the values of both patched solutions on the ball lie in a common compact interval on which $g$ has a strict sign. Lemma \ref{lem:local_signed_scp} therefore yields $u_*=u^*$ in $B_r((x_0,\overline{\lambda}))$. In particular, a negative value of $g$ is handled by the transformation $w\mapsto M-w$ in that lemma; positivity of $g$ is not tacitly assumed. That means $u=u_{\overline{\lambda}}$ in $B_r((x_0,\overline{\lambda}))$. However, this contradicts Lemma \ref{lem:scp2} with $I=(x_0-r,x_0+r)$.
		
		Hence \eqref{post_bound} holds. This implies $\frac{\partial u}{\partial y}>0$ in $B_{\overline{r}}((\overline{x},\overline{\lambda}))$ for some $\overline{r}>0$. Exploiting this fact and Lemma \ref{lem:scp2} with $I=(\overline{x}-\overline{r},\overline{x}+\overline{r})$, we find $x_1 \in (\overline{x}-\overline{r},\overline{x}+\overline{r})$ and $\tilde{\varepsilon}>0$ such that
		\begin{equation}\label{key_line0}
			\begin{aligned}
				\frac{\partial u}{\partial y}>0 &\text{ on } \{x_1\} \times (0,\overline{\lambda}+\tilde{\varepsilon}],\\
				u<u_{\lambda} &\text{ on } \{x_1\} \times [0,\lambda) \text{ for all } \lambda \in (0,\overline{\lambda}].
			\end{aligned}
		\end{equation}
		
		From the second assertion in \eqref{key_line0} we have
		\[
		u_{\overline{\lambda}} - u > C \text{ on the compact set } \{x_1\} \times [0,\overline{\lambda}-\tilde{\varepsilon}/2]
		\]
		for some $C>0$. (We may reduce $\tilde{\varepsilon}$ if necessary so that $\overline{\lambda}-\tilde{\varepsilon}/2>0$.)
		Hence there exists $0<\overline{\varepsilon} < \tilde{\varepsilon}/4$ such that
		\[
		u_\lambda - u > \frac{C}{2} > 0 \text{ on } \{x_1\} \times [0,\overline{\lambda}-\tilde{\varepsilon}/2] \text{ for all } \lambda \in [\overline{\lambda},\overline{\lambda}+\overline{\varepsilon}].
		\]
		On the other hand, exploiting the first assertion in \eqref{key_line0}, we deduce
		\[
		u_\lambda - u > 0 \text{ on } \{x_1\} \times [\overline{\lambda}-\tilde{\varepsilon}/2,\lambda) \text{ for all } \lambda \in [\overline{\lambda},\overline{\lambda}+\overline{\varepsilon}].
		\]
		
		Therefore,
		\begin{equation}\label{key_line}
			u<u_{\lambda} \text{ on } \{x_1\} \times [0,\lambda) \text{ for all } \lambda \in (0,\overline{\lambda}+\overline{\varepsilon}].
		\end{equation}
		
		We next justify that this strict inequality survives a small rotation,
		\emph{uniformly also as $s\downarrow0$}. Put
		$L:=\overline\lambda+\overline\varepsilon$. There exists
		$\theta_0>0$ such that
		\begin{equation}\label{boundary_cond'}
		u(x_1,y)<u_{x_1,s,\theta}(x_1,y)
		\quad\text{for }0<s\le L,\quad0\le y<s,\quad0<|\theta|<\theta_0.
		\end{equation}
		(The endpoint $y=s$ must be excluded: it lies on the reflection line.)
		To prove the uniform assertion, first use Theorem
		\ref{th:boundary_behavior'} to find $r,\delta>0$ and
		$\theta_1>0$ such that
		$\partial_{V_\theta}u>0$ in
		$(x_1-r,x_1+r)\times(0,4\delta)$ whenever
		$|\theta|\le\theta_1$. Choose $\delta$ smaller if needed. For
		$0<s\le 2\delta$ and $|\theta|$ sufficiently small, each normal
		segment from $(x_1,y)$ to its reflection, $0<y<s$, is contained in
		this rectangle. Integration of $\partial_{V_\theta}u>0$ along
		the segment proves \eqref{boundary_cond'} in this range; for
		$y=0$, use $u(x_1,0)=0$ and positivity at the reflected point.
		Now consider $2\delta\le s\le L$. By the first part of
		\eqref{key_line0}, and since $\overline\varepsilon<\tilde\varepsilon/4$,
		we may fix $0<\eta<\delta/2$ so small that
		$\partial_{V_\theta}u>0$ throughout
		$(x_1-r_1,x_1+r_1)\times[\delta,L+\eta]$ for some
		$r_1>0$ and all sufficiently small $|\theta|$; this follows
		from continuity of $\nabla u$ on that compact rectangle.
		For $s-\eta<y<s$, the segment joining $(x_1,y)$ to its reflection
		is a positive multiple of $V_\theta$ and, upon reducing $|\theta|$,
		stays in this rectangle. Thus the strict inequality follows
		by integration. Finally the set
		\[
		\{(s,y):2\delta\le s\le L,\ 0\le y\le s-\eta\}
		\]
		is compact. At $\theta=0$ the continuous difference
		$u(x_1,2s-y)-u(x_1,y)$ is strictly positive there by
		\eqref{key_line}, hence has a positive minimum. Uniform
		continuity of the reflection and of $u$ preserves that minimum
		for all sufficiently small $|\theta|$. Taking the minimum of
		the angle restrictions in the two ranges proves
		\eqref{boundary_cond'}.

		Combining \eqref{boundary_cond'} with the Dirichlet boundary condition of $u$, we have
		\begin{equation}\label{boundary_cond2}
			u< u_{x_1, s, \theta} \text{ on } \partial\mathcal{T}_{x_1,s,\theta} \setminus L_{x_1,s,\theta} \quad\text{ for all } 0< s\le \overline{\lambda}+\overline{\varepsilon} \text{ and } 0<|\theta|<\theta_0.
		\end{equation}
		
		With the help of \eqref{boundary_cond2}, for each $0<|\theta|<\min\{\frac{\pi}{4},\theta_0\}$, we can repeat the sliding technique in Step 1 to show that
		\[
		u\le u_{x_1,\lambda,\theta} \text{ in } \new{\mathcal{T}_{x_1,\lambda,\theta}} \quad\text{ for all } 0< \lambda\le \overline{\lambda}+\overline{\varepsilon}.
		\]
		By letting $\theta\to0$, we obtain $u\le u_{\lambda}$ in $\Sigma_{\lambda}$ for all $0<\lambda\le\overline{\lambda}+\overline{\varepsilon}$. This contradicts the definition of $\overline{\lambda}$. Thus, \eqref{s2} is proved.
		
		\textit{Step 3.} Conclusion.
		
				Step 2 implies that $u$ is monotone increasing with respect to the $y$-direction. Moreover, Lemma \ref{lem:scp2} indicates
		\[
		\frac{\partial u}{\partial y} > 0 \quad\text{ in } \mathbb{R}^2_+ \setminus Z_{g(u)}.
		\]
		
		\paragraph{The linear case $p=2$ in (i).}
        All nondegenerate local-comparison steps above are classical for
        $p=2$, including when $g$ changes sign: on any compact interior
        subset the difference of two ordered solutions satisfies a
        uniformly elliptic linear equation with bounded coefficients,
        after clipping $g$ as in Lemma \ref{lem:clip_g}. Thus their
        contact set is relatively open, and the strict inequality in
        \eqref{s1''} holds throughout the interior triangle. We replace
        the small-domain invocation in \eqref{K_ineq2} by the following
        elementary estimate. On the clipped comparison domain $D$ set
        $w=(u-v)^+\in H^1_0(D)$, where
        $v=u_{x_0,s,\theta}$. If the gradient bounds are at most $M$
        and the clipped reaction has Lipschitz constant $L_g$, testing
        the difference of the equations by $w$ gives, since $q>1$,
        \[
        \int_D|\nabla w|^2
          \le L_g\int_Dw^2+C(M,q,\vartheta)
                 \int_D|\nabla w|w
          \le\frac12\int_D|\nabla w|^2+C_0\int_Dw^2.
        \]
        The Poincar\'e inequality on sets of small measure,
        $\int_Dw^2\le C_P(|D|)\int_D|\nabla w|^2$ with
        $C_P(t)\to0$ as $t\downarrow0$, implies $w=0$ whenever
        $|D|$ is sufficiently small. All constants are uniform in $s$,
        by the fixed interior clipping above. The rest of Steps 1--3
        applies without change.

		\textbf{Case (ii):} $p>2$, $p-1\le q\le p$ and $g(t)>0$
        for $t>0$.

        We give the substitutions needed to justify the singular source,
        rather than treating it as globally Lipschitz at zero. Theorem
        \ref{th:boundary_behavior'} again supplies an initial rectangle
        where all upward normal derivatives $\partial_{V_\theta}u$
        are strictly positive, and therefore the starting comparison
        on sufficiently small triangles. Fix a tilt and let
        $\overline s$ denote the limiting sliding height as in Step 1.
        At an interior contact of $u$ and its reflection, both functions
        take values in a compact positive interval. Theorem 1.2 of
        \cite{MR3859194}, applied after the positive clipping of
        Lemma \ref{lem:clip_g}, then propagates equality locally.
        Consequently the contact set is open and closed in the connected
        triangle. It cannot be the whole triangle because of the strict
        comparison on the non-reflection sides. Hence
        \[
          u<u_{x_0,\overline s,\theta}
          \qquad\text{in }\mathcal T_{x_0,\overline s,\theta}.
        \]
        To specify the comparison constants without circularity, first
        choose a rectangle enclosing the triangles and their reflected
        images for all $s$ in a fixed small neighbourhood of
        $\overline s$, and choose $\rho>\rho'>0$ as in
        \eqref{start_deri'}--\eqref{infsup}. The collar
        $\Sigma_{\rho'}$ is handled directly by the boundary derivative
        and the positive interior lower bound, as in
        \eqref{infsup}--\eqref{K_ineq1}. Now fix a compact rectangle
        $E\subset\subset\mathbb R^2_+$ containing the closures of
        the portions of those triangles above $y=\rho'$ and all their
        reflected images. Set
        \[
        m_*:=\min_Eu>0,\qquad M_*:=\max_Eu,\qquad
        G_*:=g_{m_*,M_*},
        \]
        and fix a common upper bound for the values and gradients of
        $u$ and all the corresponding reflections on those portions.
        Since $g>0$ on $(0,\infty)$, Lemma \ref{lem:clip_g} gives a
        globally Lipschitz function $G_*>0$ which agrees with $g$
        everywhere these comparisons are made. As
        $q\ge p-1>p/2$, select $\delta_0>0$ from Theorem
        \ref{th:wcp_p>2} for the \emph{fixed} reaction $G_*$ and
        fixed bounds, before choosing any comparison domain.

        Next choose $K$ compactly contained in the interior of
        $\mathcal T_{x_0,\overline s,\theta}$, large enough that
        \[
        \bigl|\mathcal T_{x_0,\overline s,\theta}
             \setminus(K\cup\overline{\Sigma_{\rho'}})\bigr|
             <\delta_0/2.
        \]
        (Unlike Case (i), no critical-gradient set is needed.) The
        strict comparison above gives a uniform positive gap on $K$,
        preserved for $s$ slightly greater than $\overline s$.
        For such $s$, put
        \[
        D_s:=\mathcal T_{x_0,s,\theta}
             \setminus(K\cup\overline{\Sigma_{\rho'}}).
        \]
        By continuity of the moving triangle, $|D_s|<\delta_0$ for
        sufficiently small $s-\overline s>0$. On $\partial D_s$
        the comparison holds by the previously established inequality
        on $K$, the collar, and the non-reflection sides. Both functions
        solve the equation with the same $G_*$ in $D_s$ and obey the
        fixed interior bounds there. Theorem \ref{th:wcp_p>2}
        therefore gives $u\le u_{x_0,s,\theta}$ in $D_s$; combining
        this with the collar and $K$ extends the comparison throughout
        the triangle. This contradicts maximality of $\overline s$ and
        proves nonemptiness of $\Lambda$ in \eqref{s1}.

        Suppose next that $\overline\lambda:=\sup\Lambda<\infty$.
        We have $u_y\ge0$ in $\Sigma_{\overline\lambda}$.
        Positivity of $g$ permits the strong maximum principle for the
        linearized equation \cite[Theorem 1.3]{MR4022661} on each
        compact interior set, using the same local positive clipping.
        Since $u_y>0$ near the bottom boundary, exhaustion and the
        connectedness of the strip give
        $u_y>0$ throughout $\Sigma_{\overline\lambda}$.
        There must also be at least one $\bar x$ with
        $u_y(\bar x,\overline\lambda)>0$: otherwise patch the
        reflected solutions $u_*$ and $u^*$ defined in Step 2 across
        $y=\overline\lambda$. They are $C^1$ weak solutions on
        a small ball about any point of that line, have $u_*\le u^*$,
        and touch on the line. Positive local strong comparison forces
        them to coincide on that ball, contradicting
        $u_y>0$ immediately below the line. This argument uses only
        compact positive value intervals, never a Lipschitz extension
        through $u=0$.

        By continuity, $u_y>0$ on a short vertical interval above
        $(\bar x,\overline\lambda)$. For every
        $0<\tau\le\overline\lambda$, local positive strong
        comparison and the boundary values also give
        $u<u_\tau$ in $\Sigma_\tau$. Thus the strict gap at
        $\tau=\overline\lambda$ on the compact lower part of
        $\{\bar x\}\times[0,\overline\lambda)$ and the
        strict derivative near its top extend the line comparison to
        $\tau\in(0,\overline\lambda+\varepsilon]$, exactly
        as in \eqref{key_line0}--\eqref{key_line}. The uniform
        small-rotation argument in \eqref{boundary_cond'}--
        \eqref{boundary_cond2}, followed by the preceding small-area
        triangle comparison (with a fixed positive clipping), extends
        the horizontal moving plane past $\overline\lambda$.
        This contradiction completes Case (ii).
	\end{proof}
	
	\section*{Appendix}
	In this appendix, we provide a detailed proof for Theorem \ref{th:monotonicity} in case \ref{f1} holds.
	We denote $g(t) := \frac{1}{t^\gamma} + f(t)$ for $t>0$, $Z_g := \{t\in(0,+\infty) \mid g(t)=0\}$ and
	\[Z_u = \{ x \in \mathbb{R}_+^N \mid \nabla u(x) = 0 \},\]
	\[Z_{g(u)} = \{ x \in \mathbb{R}_+^N \mid g(u(x)) = 0 \}.\]
	We aim to prove the following.
	\begin{proposition}\label{prop:monotonicity}
		Assume that \ref{f1} holds and $1< q\le p$. Let $u \in W^{1,p}_{\rm loc}(\mathbb{R}^N_+)$ be a solution to problem \eqref{main} with $u\in L^\infty(\Sigma_\lambda)$ for all $\lambda>0$. Then $u$ is monotone increasing in $x_N$-direction, that is,
		\[
		\frac{\partial u}{\partial x_N} \ge 0 \quad\text{ in } \mathbb{R}^N_+.
		\]
	\end{proposition}
	
	The proof follows the techniques in \cite{MR4439897} (see also \cite{MR4902901}).
	Given \(\lambda > 0\) and a solution \(u\) to \eqref{main}, we define  
	\[
	u_\lambda(x', x_N) := u(x', 2\lambda - x_N)
	\]
	for \(x \in \Sigma_{2\lambda}\).
	We prove the following strong comparison principle.
	
	\begin{proposition}\label{prop:scp}
		Under the assumption of Proposition \ref{prop:monotonicity}, we assume
		\[
		\frac{\partial u}{\partial x_N} \geq 0 \quad \text{in } \Sigma_\lambda \quad \text{and} \quad u \leq u_\lambda \quad \text{in } \Sigma_\lambda
		\]  
		for some \(\lambda > 0\). Then \(u < u_\lambda\) in \(\Sigma_\lambda \setminus Z_u\).
	\end{proposition}
	
	\begin{proof}
				Since $g(t)>0$ for all sufficiently small $t>0$, a discrete
		zero set can be enumerated as
		$Z_g=\{z_1<z_2<\cdots\}$ (finite or infinite), with $z_0:=0$.
		When the set is finite, say $Z_g=\{z_1,\ldots,z_m\}$, put
		$z_{m+1}:=+\infty$. If it is empty, $g>0$ on $(0,+\infty)$
		by continuity; the local comparison Lemma \ref{lem:local_signed_scp}
		shows that the contact set $\{u=u_\lambda\}$ is both open and
		closed in the connected strip $\Sigma_\lambda$. It cannot be the
		whole strip because $u(x',0)=0<u(x',2\lambda)$, so there are no
		contacts and the proposition follows. Henceforth assume
		$Z_g\ne\emptyset$.
		By contradiction, assume that there exists  
		\[ 
		\overline{x} \in \Sigma_\lambda \setminus Z_u  
		\]  
		such that \( u(\overline{x}) = u_\lambda(\overline{x}) \). Since \( u \in C^1(\mathbb{R}_+^N) \), we deduce \( B_\varepsilon(\overline{x}) \subset \Sigma_\lambda \setminus Z_u \) for some \( \varepsilon > 0 \) sufficiently small. Notice that both \( u \) and \( u_\lambda \) solve  
		\[ 
		-\Delta_p w + \vartheta|\nabla w|^q = g(w) \quad \text{in } B_\varepsilon(\overline{x}).  
		\]  
		The strong comparison principle (see \cite[Theorem 2.5.2]{MR2356201}) now comes into play to yield  
		\[ 
		u = u_\lambda \quad \text{in } B_\varepsilon(\overline{x}).  
		\]  
		Moreover, since \( u \) is not constant in \( B_\varepsilon(\overline{x}) \) and \( Z_g \) is a discrete set, we can find \( x_0 \in B_\varepsilon(\overline{x}) \) such that \( x_0 \notin Z_{g(u)} \). That is,  
		\[ 
		z_k < u(x_0) < z_{k+1} \quad \text{for some } k \geq 0.  
		\] 
		
		Let \( \Omega_0 \) be the connected component of \( \Sigma_\lambda \setminus Z_{g(u)} \) containing \( x_0 \). By continuity, its relative boundary within \(\Sigma_\lambda\) lies in the level set \(u^{-1}(\{z_k,z_{k+1}\})\), omitting \(z_{k+1}\) if it is infinite. The contact set \(\{u=u_\lambda\}\cap\Omega_0\) is closed in \(\Omega_0\). It is also relatively open: at each of its points both solution values are equal and belong to a compact subinterval of \( (z_k,z_{k+1})\), on which \(g\) has a strict sign, so that Lemma \ref{lem:local_signed_scp} applies on a small ball. Connectedness and the contact at \(x_0\) consequently give  
		\begin{equation}\label{Omega0}
			u = u_\lambda \quad \text{in } \Omega_0.
		\end{equation}  
		Because \( \Omega_0 \) is open, there exists \( r_0 > 0 \) such that  
		\[ 
		B_{2r_0}(x_0) \subset \Omega_0.
		\]
		
		Now we slide \( B_{r_0}(x_0) \) in \( \Omega_0 \), towards \( \partial \mathbb{R}_+^N \) in direction \(-e_N \) and keep its center on the ray \(\{ x_0 - te_N \mid t \geq 0 \}\), until it touches for the first time \( \partial \Omega_0 \) at some point \( \hat{x}_0 \in \partial \Omega_0 \). We denote by \( \tilde{x}_0 = x_0 - t_0 e_N \) the new center of the slid ball.
		
		Since \( u \) is monotone non-decreasing in \( \Sigma_\lambda \), for all \( x \in B_{r_0}(\tilde{x}_0) \), we have
		\[
		z_k < u(x) \leq u(x + t_0 e_N) \leq \max_{B_{r_0}(x_0)} u < z_{k+1}.
		\]
		
		Therefore, the first touching point \( \hat{x}_0 \) on the relative boundary of \(\Omega_0\) must satisfy \( u(\hat{x}_0) = z_k \), or lie on \(\partial\mathbb R^N_+\) where \(u(\hat x_0)=0=z_0\). A first contact with the top of \(\Sigma_\lambda\) cannot occur, since the ball is translated only downward from a ball compactly contained in the strip. Moreover, by continuity, we deduce from \eqref{Omega0} that \( u(\hat{x}_0) = u_\lambda(\hat{x}_0) \). We consider two cases.
		
		\textit{Case 1:} \( k = 0 \), i.e., \( u(\hat{x}_0) = z_0 = 0 \). Then \new{\( u(\hat{x}_0', 2\lambda - \hat{x}_{0,N}) = u_\lambda(\hat x_0) = 0 \)}. This contradicts the fact that \( u \) is positive in \( \mathbb{R}_+^N \).
		
		\textit{Case 2:} \( k \geq 1 \). Let us define the function
		\[
		w(x) := u(x) - z_k > 0 \quad \text{for } x \in B_{r_0}(\tilde{x}_0).
		\]
		
		Since \new{\( 1 < p \le 2 \)} and \( g \) is Lipschitz continuous in any compact interval $K\subset (0,+\infty)$, we have
		\[
		Cw^{p-1} + g(u) = Cw^{p-1} + g(u) - g(z_k) \geq Cw^{p-1} - C_0(u - z_k) \geq 0
		\]
		for sufficiently large \( C \). Hence \( w \) satisfies
		\[
		\begin{cases}
			-\Delta_p w + \vartheta|\nabla w|^q + Cw^{p-1} \geq 0 & \text{in } B_{r_0}(\tilde{x}_0), \\
			w > 0 & \text{in } B_{r_0}(\tilde{x}_0), \\
			w(\hat{x}_0) = 0.
		\end{cases}
		\]
		
		By the Hopf boundary lemma (see Theorems 2.5.1 and 5.5.1 in \cite{MR2356201}), we have
		\[
		\frac{\partial u}{\partial \eta}(\hat{x}_0) < 0,
		\]
		where \(\eta = \frac{\hat{x}_0 - \tilde{x}_0}{|\hat{x}_0 - \tilde{x}_0|}\) is the outward normal at \(\hat{x}_0\). In particular, \(|\nabla u(\hat{x}_0)| \neq 0\). Since \(u \in C^1(\mathbb{R}^N_+)\), there exists a ball \(B_{\rho_0}(\hat{x}_0) \subset \Sigma_\lambda\) such that \(|\nabla u| \neq 0\) in \(B_{\rho_0}(\hat{x}_0)\). By the strong comparison principle, since \(u(\hat{x}_0) = u_\lambda(\hat{x}_0)\), we have	
		\[
		u = u_\lambda \quad \text{in } B_{\rho_0}(\hat{x}_0).
		\]
		
		\new{From the last display}, we can find a point \(x_1 \in \{\hat{x}_0 + t\eta \mid t > 0\} \cap B_{\rho_0}(\hat{x}_0)\) which is close to \(\hat{x}_0\) such that		
		\[
		z_{k-1} < u(x_1) < u(\hat{x}_0) = z_k.
		\]
		
		Therefore, from a point \(x_0 \in \Sigma_\lambda\) with \(u(x_0) = u_\lambda(x_0)\) and \(z_k < u(x_0) < z_{k+1}\), we have found a new point \(x_1 \in \Sigma_\lambda\) satisfying \(u(x_1) = u_\lambda(x_1)\) and \(z_{k-1} < u(x_1) < z_k\). Repeating this argument a finite number of times, we finally reach a point \(x_k \in \Sigma_\lambda\) such that \(u(x_k) = u_\lambda(x_k)\) and \(z_0 < u(x_k) < z_1\). Then we reach a contradiction as in Case 1.
	\end{proof}
	
	We also recall the following weak comparison principle\rev{, which we state in the form in which it is used below. Since the operator is invariant under translations, the strip $\Sigma_{\lambda_0}$ may be replaced by any strip $\Sigma_{a,a+\lambda_0}$, $a\ge0$.}
	\begin{theorem}[Weak comparison principle in strips \cite{MR3303939}]\label{th:wcp}
		Let \(1 < p < 2\), \(1 < q \le p\) \new{and let} \new{\(h\)} be a locally Lipschitz continuous function on \(\mathbb{R}\). Fix \(\lambda_0 > 0\) and \(L_0 > 0\). \rev{Let the open set \(\mathcal{S}_{\tau,\varepsilon} \subset \Sigma_{\lambda_0}\), where \(\tau, \varepsilon > 0\), be such that
		\[
		\mathcal{S}_{\tau,\varepsilon} = \bigcup_{x' \in \mathbb{R}^{N-1}} \{x'\} \times (A^{\tau}_{x'} \cup B^{\varepsilon}_{x'}), \quad A^{\tau}_{x'} \cap B^{\varepsilon}_{x'} = \emptyset,
		\]
		where, for each \(x'\), \(A^{\tau}_{x'}, B^{\varepsilon}_{x'} \subset (0, \lambda_0)\) are measurable sets. Let \(u, v \in C^1(\overline{\mathcal{S}_{\tau,\varepsilon}})\) be such that
		\[
		\|u\|_{L^\infty(\mathcal{S}_{\tau,\varepsilon})} + \|\nabla u\|_{L^\infty(\mathcal{S}_{\tau,\varepsilon})} \leq L_0, \quad \|v\|_{L^\infty(\mathcal{S}_{\tau,\varepsilon})} + \|\nabla v\|_{L^\infty(\mathcal{S}_{\tau,\varepsilon})} \leq L_0,
		\]
		\[
		|A^{\tau}_{x'}| \leq \tau \quad \text{and} \quad B^{\varepsilon}_{x'} \subset \{x_N \in (0, \lambda_0) \mid |\nabla u(x', x_N)| \leq \varepsilon, |\nabla v(x', x_N)| \leq \varepsilon\}
		\]
		and
		\[
		\begin{cases}
			-\Delta_p u + \vartheta|\nabla u|^q \leq h(u) & \text{in } \mathcal{S}_{\tau,\varepsilon}, \\
			-\Delta_p v + \vartheta|\nabla v|^q \geq h(v) & \text{in } \mathcal{S}_{\tau,\varepsilon}, \\
			u \leq v & \text{on } \partial\mathcal{S}_{\tau,\varepsilon}.
		\end{cases}
		\]
		Then there exist
		\[
		\tau_0 = \tau_0(N, p, q, \vartheta, h, \lambda_0, L_0) > 0
		\quad\text{and}\quad
		\varepsilon_0 = \varepsilon_0(N, p, q, \vartheta, h, \lambda_0, L_0) > 0
		\]
		such that, if \(0 < \tau < \tau_0\) and \(0 < \varepsilon < \varepsilon_0\), it follows that
		\[
		u \leq v \quad \text{in } \mathcal{S}_{\tau,\varepsilon}.
		\]}
	\end{theorem}
	
	We are ready to prove Proposition \ref{prop:monotonicity}.
	
	\begin{proof}[Proof of Proposition \ref{prop:monotonicity}]
		We start moving the plane from \(\{x_N = 0\}\) to \(+\infty\). Let us consider	
		\[
		\Lambda := \{\lambda > 0 \mid u \leq u_t \text{ in } \Sigma_t \text{ for all } 0 < t \leq \lambda\}.
		\]
		By Theorem \ref{th:boundary_behavior}, we have \(\Lambda \neq \emptyset\). Hence, we can define	
		\(
		\overline{\lambda} = \sup \Lambda.
		\)
		To conclude the proof, we have to show that	\(\overline{\lambda} = +\infty\).
		By contradiction, in what follows, we assume that \(\overline{\lambda} < +\infty\). Then we have \(u \leq u_{\overline{\lambda}} \text{ in } \Sigma_{\overline{\lambda}}\).
		
		\rev{We first fix the height of the region near \(\partial\mathbb{R}^N_+\) that we discard. By Theorem \ref{th:boundary_bound} and Lemma \ref{lem:global_lower_bound}, there is \(\delta_* \in (0, \frac{\overline{\lambda}}{2})\) such that
		\begin{equation}\label{boundary}
			\sup_{\Sigma_{\delta_*}} u \le \inf_{\mathbb{R}^N_+\setminus\Sigma_{\overline{\lambda}}} u.
		\end{equation}
		Set \(\lambda_0 := \overline{\lambda} + 1\), \(m:=\inf_{\mathbb{R}^N_+\setminus\Sigma_{\delta_*/2}}u>0\), \(M:=\sup_{\Sigma_{2\lambda_0}}u\) and
		\[
		L_0 := \|u\|_{L^{\infty}(\Sigma_{\delta_*/2,2\lambda_0})} + \|\nabla u\|_{L^{\infty}(\Sigma_{\delta_*/2,2\lambda_0})}.
		\]
		Here \(m>0\) by Lemma \ref{lem:global_lower_bound}, and \(L_0<+\infty\) because \(g(u)\) is bounded in \(\Sigma_{\delta_*/4,2\lambda_0+1}\), so that the interior gradient estimates \cite{MR709038,MR727034} on balls of radius \(\frac{\delta_*}{4}\) apply uniformly. Let \(h\) be the function which equals \(g\) on \([m,M]\) and is constant on \((-\infty,m]\) and on \([M,+\infty)\). Then \(h\) is Lipschitz continuous on \(\mathbb{R}\) and \(h(u)=g(u)\), \(h(u_\lambda)=g(u_\lambda)\) in \(\Sigma_{\delta_*,\lambda}\) for \(\lambda\le\lambda_0\). Take \(\tau_0\), \(\varepsilon_0\) as in Theorem \ref{th:wcp}, applied in the strip \(\Sigma_{\delta_*,\delta_*+\lambda_0}\) with these \(h\), \(\lambda_0\), \(L_0\). Note that, for \(\lambda\le\lambda_0\), both \(u\) and \(u_\lambda\) satisfy the bound \(L_0\) in \(\Sigma_{\delta_*,\lambda}\), since the reflection of \(\Sigma_{\delta_*,\lambda}\) with respect to \(\{x_N=\lambda\}\) is \(\Sigma_{\lambda,2\lambda-\delta_*}\subset\Sigma_{\delta_*/2,2\lambda_0}\).
		
		Now fix \(0 < \delta < \min\left\{\frac{\overline{\lambda}}{2}, \frac{\tau_0}{4}\right\}\) and \(0 < \rho < \varepsilon_0\); thus \(\delta_*<\overline\lambda-\delta\). For \(\varepsilon>0\) we set
		\[
		K_\varepsilon := \left\{x\in\mathbb{R}^N \mid \delta_*\le x_N\le\overline{\lambda}-\delta,\ |\nabla u(x)|\ge\rho \text{ or } |\nabla u_{\overline{\lambda}+\varepsilon}(x)|\ge\rho\right\}.
		\]
		We claim that there exists \(\overline{\varepsilon} > 0\) such that for all \(0 < \varepsilon \leq \overline{\varepsilon}\) it follows that	
		\begin{equation}\label{compact}
			u \leq u_{\overline{\lambda}+\varepsilon} \quad \text{in } K_\varepsilon.
		\end{equation}
		Indeed, if \eqref{compact} does not hold, then we can find a sequence of positive numbers \((\varepsilon_n)\) converging to 0 and a sequence of points \(x_n = (x'_n, y_n) \in K_{\varepsilon_n}\) such that	
		\[
		u(x_n) > u_{\overline{\lambda}+\varepsilon_n}(x_n) \quad \text{and} \quad |\nabla u(x_n)| \geq \rho \ \text{ or } \ |\nabla u_{\overline{\lambda}+\varepsilon_n}(x_n)| \geq \rho
		\]	
		for all \(n\). Passing to a subsequence if necessary, we may and do assume that \(\lim_{n \to \infty} y_n = y_0 \in [\delta_*, \overline{\lambda} - \delta]\). Let us define	
		\[
		u_n(x', y) = u(x' + x'_n, y) \quad \text{for } (x', y) \in \mathbb{R}_+^N.
		\]
		Then \( u_n \) satisfies
		\begin{equation}\label{un_eq}
			\begin{cases}
				-\Delta_p u_n + \vartheta|\nabla u_n|^q = g(u_n) & \text{in } \mathbb{R}^N_+, \\
				u_n = 0 & \text{on } \partial\mathbb{R}^N_+.
			\end{cases}
		\end{equation}
		Since \(u\) is bounded in every strip, and since the lower bound \eqref{t1_est} and the upper bounds of Theorem \ref{th:boundary_bound} do not depend on \(x'\), the sequence \((u_n)\) is bounded from above and from below by positive constants on \(\Sigma_{a,b}\) for \(b>a>0\), uniformly in \(n\), and \(u_n\to0\) as \(x_N\to0^+\) uniformly in \(x'\) and \(n\).
		By \( C^{1,\alpha} \) estimates, Ascoli--Arzel\`a's theorem and a standard diagonal process, we deduce
		\[
		u_n \to \tilde{u} \quad \text{in } C^{1,\alpha'}_{\rm loc}(\mathbb{R}^N_+)
		\]
		up to a subsequence for \( 0 < \alpha' < \alpha \), where \(\tilde u\) is a solution to \eqref{main} which is bounded in every strip. Moreover, with \( x_0 := (0, y_0) \in \Sigma_{\overline{\lambda}} \), we have
		\begin{enumerate}[label=(\roman*)]
			\item \( \tilde{u} \leq \tilde{u}_\lambda \) in \( \Sigma_\lambda \) for all \( 0 < \lambda \leq \overline{\lambda} \),
			\item \( \tilde{u}(x_0) \geq \tilde{u}_{\overline{\lambda}}(x_0) \),
			\item \( |\nabla\tilde{u}(x_0)| \geq \rho \) or \( |\nabla\tilde{u}_{\overline\lambda}(x_0)| \geq \rho \).
		\end{enumerate}
		Item (i) implies that \( \frac{\partial \tilde{u}}{\partial x_N} \geq 0 \) in \( \Sigma_{\overline{\lambda}} \). Items (i) and (ii) imply that \( \tilde{u}(x_0) = \tilde{u}_{\overline{\lambda}}(x_0) \); hence the nonnegative function \(\tilde u_{\overline\lambda}-\tilde u\) attains its minimum \(0\) at the interior point \(x_0\), so that \(\nabla\tilde u(x_0)=\nabla\tilde u_{\overline\lambda}(x_0)\), and (iii) gives $x_0\notin Z_{\tilde u}$. However, the existence of such a solution \( \tilde{u} \) and such a point \( x_0 \) contradicts Proposition \ref{prop:scp}. Therefore, \eqref{compact} must hold.
		
		Let \(0<\varepsilon\le\min\{\overline\varepsilon,\frac{\tau_0}{2},1\}\), \(\lambda:=\overline\lambda+\varepsilon\) and \(v:=u_\lambda\). We apply Theorem \ref{th:wcp} in the strip \(\Sigma_{\delta_*,\delta_*+\lambda_0}\) to the open set
		\[
		\mathcal{S} := \Sigma_{\overline{\lambda}-\delta,\lambda} \cup \left\{x\in\Sigma_{\delta_*,\overline\lambda-\delta/2} \mid |\nabla u(x)|<\rho,\ |\nabla v(x)|<\rho\right\},
		\]
		with \(A_{x'}:=(\overline\lambda-\delta,\lambda)\) and \(B_{x'}:=\{x_N\in(\delta_*,\overline\lambda-\delta] \mid |\nabla u(x',x_N)|<\rho,\ |\nabla v(x',x_N)|<\rho\}\). Then \(|A_{x'}|=\delta+\varepsilon<\tau_0\) and \(\rho<\varepsilon_0\). We check that \(u\le v\) on \(\partial\mathcal{S}\). On \(\{x_N=\lambda\}\) we have \(u=v\). On \(\{x_N=\delta_*\}\) the inequality follows from \eqref{boundary}, because \(v(x)=u(x',2\lambda-\delta_*)\) with \(2\lambda-\delta_*>\overline\lambda\). At the remaining boundary points we have \(\delta_*<x_N\le\overline\lambda-\delta\) and \(|\nabla u(x)|\ge\rho\) or \(|\nabla v(x)|\ge\rho\), so that \(x\in K_\varepsilon\) and \eqref{compact} applies. Theorem \ref{th:wcp} yields \(u\le v\) in \(\mathcal{S}\).
		
		Since \(\Sigma_\lambda\subset\Sigma_{\delta_*}\cup K_\varepsilon\cup\mathcal{S}\) up to the hyperplane \(\{x_N=\delta_*\}\), on which \eqref{boundary} applies, we deduce from \eqref{boundary}, \eqref{compact} and the last conclusion that \( u \leq u_{\overline{\lambda}+\varepsilon} \) in \( \Sigma_{\overline{\lambda}+\varepsilon} \) for all \( 0 < \varepsilon \leq \min\left\{\overline{\varepsilon}, \frac{\tau_0}{2}, 1\right\} \).} This is a contradiction with the definition of \( \overline{\lambda} \). Therefore, it must hold that \( \overline{\lambda} = +\infty \). This implies the monotonicity of \( u \) in the \( x_N \)-direction.
	\end{proof}
	
	\rev{\begin{remark}\label{rem:p=2}
		Theorem \ref{th:wcp} is stated for \(1<p<2\), while \ref{f1} allows \(p=2\). When \(p=2\) the proof above simplifies. For \(\lambda\le\overline\lambda\), the function \(w:=u_\lambda-u\ge0\) satisfies in \(\Sigma_\lambda\) a linear equation \(-\Delta w+b\cdot\nabla w-c\,w=0\) whose coefficients are bounded on compact subsets of \(\Sigma_\lambda\), because \(q>1\) and \(g\) is locally Lipschitz continuous on \((0,+\infty)\). By the classical strong maximum principle, either \(w>0\) in \(\Sigma_\lambda\) or \(w\equiv0\), and the latter is impossible since \(u(x)\to0\) while \(u_\lambda(x)\to u(x',2\lambda)>0\) as \(x_N\to0^+\). Thus Proposition \ref{prop:scp} holds with \(Z_u\) replaced by the empty set. Consequently, \eqref{compact} holds with \(K_\varepsilon\) replaced by \(\overline{\Sigma_{\delta_*,\overline\lambda-\delta}}\), and one only needs the comparison in the narrow strip \(\Sigma_{\overline\lambda-\delta,\lambda}\). There the difference \(u-v\) satisfies a linear inequality with bounded coefficients, and the maximum principle for linear operators with bounded coefficients in narrow, possibly unbounded, domains applies (see, e.g., \cite{MR1655510}).
	\end{remark}}
	
	\section*{Statements and Declarations}
	
	
	\textbf{Data Availability} Data sharing not applicable to this article as no datasets were generated or analysed during the current study.
	
	\textbf{Competing Interests} The author has no competing interests to declare that are relevant to the content of this article.
	
	\section*{Acknowledgments}
	The author conducted this work while visiting the Vietnam Institute for Advanced Study in Mathematics (VIASM) in 2025 and thanks the institute for its warm hospitality and research support.
	
	\bibliographystyle{abbrvurl}
	\bibliography{references}
	
\end{document}